\documentclass[10pt,reqno]{amsart}
   \usepackage[centertags]{amsmath}
   \usepackage{amsfonts}
   \usepackage{amsmath}
   \usepackage{amssymb}
   \usepackage{amsthm}
   \usepackage{newlfont}
   \usepackage[latin1]{inputenc}

\usepackage{multirow, bigdelim}

\usepackage{epsfig}
\usepackage{pst-all}
\usepackage{pstricks}
\usepackage{pgf}
\usepackage{mathrsfs}
\usepackage{hyperref}
\usepackage{bm}

\usepackage{graphicx}

\usepackage{bm}

\usepackage{mathtools}
\usepackage{ifmtarg}

\makeatletter

    \newcommand\contFrac{\@ifstar{\@contFracStar}{\@contFracNoStar}}

    \def\singleContFrac#1#2{%
        \begin{array}{@{}c@{}}%
            \multicolumn{1}{c|}{#1}%
            \\%
            \hline%
            \multicolumn{1}{|c}{#2}%
        \end{array}%
    }

    \def\@contFracNoStar#1{%
        \mathchoice{
            \@contFracNoStarDisplay@#1//\@nil%
        }{
            \@contFracNoStarInline@#1//\@nil%
        }{
            \@contFracNoStarInline@#1//\@nil%
        }{
            \@contFracNoStarInline@#1//\@nil%
        }%
    }

    \def\@contFracNoStarDisplay@#1//#2\@nil{%
        \@ifmtarg{#2}{%
            #1%
        }{%
            #1+\cfrac{1}{\@contFracNoStarDisplay@#2\@nil}%
        }%
    }

        \def\@contFracNoStarInline@#1//#2\@nil{%
            \@ifmtarg{#2}{%
                #1%
            }{%
                #1 \@@contFracNoStarInline@@#2\@nil%
            }%
        }
        \def\@@contFracNoStarInline@@#1//#2\@nil{%
            \@ifmtarg{#2}{%
                + \singleContFrac{1}{#1}%
            }{%
                + \singleContFrac{1}{#1} \@@contFracNoStarInline@@#2\@nil%
            }%
        }

    \def\@contFracStar#1{%
        \mathchoice{
            \@contFracStarDisplay@#1////\@nil%
        }{
            \@contFracStarInline@#1//\@nil%
        }{
            \@contFracStarInline@#1//\@nil%
        }{
            \@contFracStarInline@#1//\@nil%
        }%
    }

    \def\@contFracStarDisplay@#1//#2//#3\@nil{%
        \@ifmtarg{#2}{%
            #1%
        }{%
            #1 + \cfrac{#2}{\@contFracStarDisplay@#3\@nil}%
        }%
    }

        \def\@contFracStarInline@#1//#2\@nil{%
            \@ifmtarg{#2}{%
                #1%
            }{%
                #1 \@@contFracStarInline@@#2\@nil%
            }%
        }
        \def\@@contFracStarInline@@#1//#2//#3\@nil{%
            \@ifmtarg{#3}{%
                - \singleContFrac{#1}{#2}%
            }{%
                - \singleContFrac{#1}{#2} \@@contFracStarInline@@#3\@nil%
            }%
        }
\makeatother

\allowdisplaybreaks[3]

\theoremstyle{plain}
\newtheorem{theorem}{Theorem}[section]
\newtheorem{lemma}[theorem]{Lemma}
\newtheorem{proposition}[theorem]{Proposition}

\theoremstyle{definition}

\theoremstyle{remark}
\newtheorem{remark}[theorem]{Remark}
\newtheorem*{remark*}{Remark}

\numberwithin{equation}{section}

\newcommand\D{\displaystyle}

\newcommand\ZZ{{\mathbb Z}}

\newcommand\p{\mbox{$\mathfrak{p}$}}

\title[Bilateral birth-death chains and associated Jacobi polynomials]{The bilateral birth-death chain generated by \\ the associated Jacobi polynomials}

\author{Manuel D. de la Iglesia}
\address{Manuel D. de la Iglesia\\
Instituto de Matem\'aticas, Universidad Nacional Aut\'onoma de M\'exico, Circuito Exterior, C.U., 04510, Ciudad de M\'exico, M\'exico.}
\email{mdi29@im.unam.mx}

\author{Claudia Juarez}
\address{Claudia Juarez\\
Instituto de Investigaciones en Matem\'aticas Aplicadas y en Sistemas, Universidad Nacional Aut\'onoma de M\'exico, Circuito Escolar 3000, C.U., 04510, Ciudad de M\'exico, M\'exico.}
\email{claudiajrz@sigma.iimas.unam.mx}
\date{\today}
\thanks{
}

\thanks{This work was partially supported by PAPIIT-DGAPA-UNAM grant IN106822 (M\'exico) and CONACYT grant A1-S-16202 (M\'exico).}

\date{\today}



\begin{document}

\maketitle

\begin{abstract}
We give a probabilistic interpretation of the associated Jacobi polynomials, which can be constructed from the three-term recurrence relation for the classical Jacobi polynomials by shifting the integer index $n$ by a real number $t$. Under certain restrictions, this will give rise to a doubly infinite tridiagonal stochastic matrix which can be interpreted as the one-step transition probability matrix of a discrete-time bilateral birth-death chain with state space on $\mathbb{Z}$. We also study the unique UL and LU stochastic factorizations of the transition probability matrix, as well as the discrete Darboux transformations and corresponding spectral matrices. Finally, we use all these results to provide an urn model on the integers for the associated Jacobi polynomials.
\end{abstract}

\section{Introduction}

The connection between discrete or continuous time birth-death chains and orthogonal polynomials is very well-known. This started in a series of papers by S. Karlin and J. McGregor \cite{KMc2, KMc3, KMc6} in the 1950s, where they established an important connection between the corresponding transition probability functions by means of a spectral representation, the so-called \emph{Karlin-McGregor integral representation formula}. This formula arises from the fact that the one-step transition probability matrix $P$ of a discrete-time birth-death chain is a tridiagonal (or Jacobi) matrix, so we can apply Favard's or the spectral theorem to find the corresponding spectral measure associated with the process. Many probabilistic aspects can be analyzed in terms of the corresponding orthogonal polynomials, such as $n$-step transition probabilities, the invariant measure or recurrence of the process, among other applications. In the last 60 years, many authors like M.E.H. Ismail, G. Valent, P. Flajolet, F. Guillemin, H. Dette or E. van Doorn, to mention a few, have studied this connection and other probabilistic aspects. For a detailed account of all these relations see the recent monograph \cite{MDIB}.

The spectral analysis of birth-death chains is typically performed on the state space of nonnegative integers $\mathbb{Z}_{\geq0}$. However, there are some situations in physics, chemistry or engineering where the state space is the set of all integers $\mathbb{Z}$ (see \cite{Con,dCIM,PL,TB}). These processes are usually known as \emph{bilateral birth-death chains} (other names like unrestricted birth-death chains or double-ended systems can also be found in the literature). The spectral analysis of discrete-time bilateral birth-death chains was firstly studied by S. Karlin and J. McGregor in the last section of \cite{KMc6}, while the case of continuous-time was considered later by W.E. Pruitt in \cite{PruT,Pru}. In this case the one-step transition probability matrix $P$ is a doubly infinite Jacobi matrix. The application of the spectral theorem will give rise now to a $2\times2$ matrix of measures which is usually called the \emph{spectral matrix} of the bilateral birth-death chain. As in the case of regular birth-death chains there will be an integral representation of the $n$-step transition probability matrix $P^{(n)}$ in terms of this spectral matrix and the corresponding (two families of linearly independent) orthogonal polynomials.

Unlike the case of birth-death chains on $\mathbb{Z}_{\geq0}$, where the spectral measure and the corresponding orthogonal polynomials of many examples have been studied  (see \cite[Chapter 2]{MDIB}), there are few explicit examples of bilateral birth-death chains where the spectral matrix and the corresponding orthogonal polynomials have been explicitly computed (see \cite{KMc6,G1,dIJ1,dIJ2} for discrete-time examples and \cite{ILMV,dI3} for continuous-time examples). The examples studied so far are variations of random walks on $\mathbb{Z}$, where the transition probabilities are constant. The purpose of this paper is to study for the first time the spectral analysis of a nontrivial discrete-time bilateral birth-death chain, where the transition probabilities are rational functions depending on the state of the chain, instead of constants. This chain will be given by the doubly infinite Jacobi matrix for the so-called associated Jacobi polynomials.

The subject of associated orthogonal polynomials (for $n\in\mathbb{Z}_{\geq0}$) is very classical. It was first considered in \cite{BarDic,Pal} and then in a paper by D. Askey and J. Wimp \cite{AWi}, where the spectral measures for associated Laguerre and Hermite polynomials were computed. The case of associated Jacobi polynomials was considered by J. Wimp in \cite{Wi}. Later works on the subject can also be found in \cite{ILVW,IM,IR}. The idea is very simple: consider the polynomials satisfying the three-term recurrence relation for the classical families of Hermite, Laguerre, Jacobi or Bessel polynomials, but replacing $n$ by $n+t$, where $t$ is an arbitrary real parameter. This construction can be extended to the case where $n\in\mathbb{Z}$, in which case we have a doubly infinite Jacobi matrix. This was considered by F.A. Gr\"unbaum and L. Haine in \cite{GH0}, where all the solutions of the \emph{bispectral problem} or Bochner's problem (see Remark \ref{rembis} below for more information) are obtained for the doubly infinite tridiagonal matrices corresponding to the associated Hermite, Laguerre, Jacobi or Bessel polynomials, and the corresponding differential operator has order two. In \cite{GH0} an extra parameter $t$ appears by allowing the translation of $n$ by $n+t$. These are the simplest evolutions of these objects away from their value at $t=0$.

Although a complete classification of the bispectral problem for this extension of the Bochner's problem was given in \cite{GH0}, very little is known about the spectral matrices for these associated polynomials. For instance in \cite{MR} some of this topics are considered but no computation of the spectral matrix was given. However, in \cite[Theorem 1]{GH}, F.A. Gr\"unbaum and L. Haine managed to compute the explicit expression of the spectral matrix for the associated Jacobi polynomials supported on $[0,1]$ for some special choice of the parameters involved. This will be the starting point of our paper. The fact that we already have an explicit expression of this spectral matrix will be the key to have a Karlin-McGregor representation formula for the $n$-step transition probability matrix of the corresponding bilateral birth-death chain. In the last few years, some more examples of doubly infinite Jacobi matrices have been analyzed (see \cite{IsmS,DIW,dIJ1,dIJ2,dI3}), computing the corresponding spectral matrices.

This paper is organized as follows. In Section \ref{sec2} we will recall some of the results contained in \cite{GH}, especially the explicit expression of the spectral matrix for the associated Jacobi polynomials. In Section \ref{sec3} we will turn the doubly infinite matrix given in \cite{GH} into a stochastic matrix $P$. For that we will analyze in Proposition \ref{prop1} under what circumstances we can choose the real parameter $t$ such that $P$ is a stochastic matrix. A special remark about the bispectrality of the matrix-valued orthogonal polynomials constructed from $P$ will be given at the end of this section. In Section \ref{sec4} we will use some results of \cite{dIJ1} (see also \cite{GdI3}) in order to obtain stochastic UL and LU factorizations of the transition probability matrix $P$. These factorizations typically come with a free parameter (see \cite{dIJ1}), but we will show that, in the case of associated Jacobi polynomials, these stochastic factorizations are \emph{unique} under certain assumption of the parameters (see Theorems \ref{thm1} and \ref{thm2}). In Section \ref{sec5} we consider discrete Darboux transformations of the UL or LU factorizations of $P$, consisting of inverting the order of the factors. Using again some results of \cite{dIJ1} we will be able to identify the spectral matrices associated with these Darboux transformations in terms of the so-called \emph{Geronimus transformations} of the original spectral matrix. Finally, in Section \ref{sec6}, we give an urn model for the associated Jacobi polynomials. This model is different from the ones considered for instance in \cite{G4,GdI3,GdI4}, since in this case we will allow the state space of the urn model to run along the whole set of integers $\mathbb{Z}$.

\section{The spectral matrix of the associated Jacobi polynomials}\label{sec2}
In this section we recall some of the results obtained F.A. Gr\"unbaum and L. Haine in \cite{GH}, where they showed the connection between a solution of a discrete-continuous version of the bispectral problem with the concept of associated polynomials. They also obtained an explicit expression of the spectral matrix corresponding to the associated Jacobi polynomials, a tool that will be extensively used in the rest of the paper. 

In \cite{GH}, the authors use parameters $a,b,c$ (according to the standard notation of the Gauss hypergeometric equation) with certain restrictions. In this paper we will use parameters $\alpha,\beta,t$ (a notation more appropriate for Jacobi polynomials) subject to the following relations
\begin{equation*}
\begin{split}
a&=\alpha+\beta+t+1,\\
b&=-t,\\
c&=\alpha+1.
\end{split}
\end{equation*}
For $n\in\mathbb{Z}$ let
\begin{equation}\label{TTRRcn}
\begin{split}
a_n&=\frac{(n+t)(n+t+\alpha)(n+t+\beta)(n+t+\alpha+\beta)}{(2n+2t+\alpha+\beta-1)(2n+2t+\alpha+\beta)^2(2n+2t+\alpha+\beta+1)},\\
b_n&=\frac{1}{2}\left(1+\frac{\alpha^2-\beta^2}{(2n+2t+\alpha+\beta-2)(2n+2t+\alpha+\beta)}\right).
\end{split}
\end{equation}
From these coefficients we can define two families of linearly independent polynomials $(P_n^\eta)_{n\in\mathbb{Z}}, \eta=1,2,$ by the three-term recurrence relation with initial conditions
\begin{align*}
\nonumber P_0^1(x)&=1,\quad P_{0}^2(x)=0,\\
\label{TTRRZP}P_{-1}^1(x)&=0,\quad P_{-1}^2(x)=1,\\
\nonumber xP_n^{\eta}(x)&=\sqrt{a_{n+1}}P_{n+1}^{\eta}(x)+b_{n+1}P_n^{\eta}(x)+\sqrt{a_n}P_{n-1}^{\eta}(x),\quad n\in\ZZ,\quad\eta=1,2.
\end{align*}
These two families of polynomials form the eigenvectors of the eigenvalue equation $ x p^\eta(x) = J p^\eta(x) $ where $p^\eta(x)=(\cdots, P_{-1}^\eta(x),P_0^\eta(x),P_1^\eta(x),\cdots)^T, \eta=1,2$, and $J$ is the doubly infinite Jacobi matrix 
\begin{equation}\label{JZ}
J=
\left(
\begin{array}{cccc|cccc}
\ddots&\ddots&\ddots&&&\\
&\sqrt{a_{-2}}&b_{-1}&\sqrt{a_{-1}}&&&\\
&&\sqrt{a_{-1}}&b_{0}&\sqrt{a_0}&&&\\
\hline
&&&\sqrt{a_0}&b_1&\sqrt{a_1}&&\\
&&&&\sqrt{a_1}&b_2&\sqrt{a_2}&\\
&&&&&\ddots&\ddots&\ddots
\end{array}
\right).
\end{equation}
This operator $J$ was called in \cite{GH} the \emph{associated Jacobi matrix}. In order to apply the spectral theorem we need to assume that $a_n>0$ for all $n\in\mathbb{Z}$. As it was pointed out in \cite{GH} it is enough to assume that $\alpha$ and $\beta$ are inside the square $-1<\alpha,\beta<1$. In that case it will be possible to find some $t$ such that positivity is ensured (more on this in the next section). Outside of this square it is possible to see that there is no value of $t$ that makes all $a_n$ positive. In \cite[Theorem 1]{GH} it is computed explicitly the $2\times2$ spectral matrix $\bm W$ supported on $[0,1]$ such that $(P_n^\eta)_{n\in\mathbb{Z}}, \eta=1,2$ are \emph{orthonormal} in the following sense
$$
\int_0^1\begin{pmatrix} P_n^1(x),&\hspace{-0.3cm}P_n^2(x)\end{pmatrix}\bm W(x)\begin{pmatrix} P_m^1(x)\\P_m^2(x)\end{pmatrix}dx=\delta_{nm},\quad n,m\in\mathbb{Z}.
$$
Under the conditions on the parameters $\alpha,\beta,t$ that ensures positivity of all $a_n$, it turns out that the Jacobi matrix $J$ has no discrete spectrum. Therefore the spectral matrix will be given by only an absolutely continuous part 
\begin{equation}\label{WW}
\bm W(x)=x^\alpha(1-x)^\beta
\begin{pmatrix} \Sigma_{11}(x)&\Sigma_{12}(x)\\\Sigma_{12}(x)&\Sigma_{22}(x)\end{pmatrix},\quad x\in[0,1],
\end{equation}
where
\begin{equation*}
\begin{split}
\Sigma_{11}(x)&=\gamma L\left(G_1^2(x)-\mu K^2x^{-2\alpha}G_2^2(x)\right),\\
\Sigma_{12}(x)&=-L\left(G_1(x)G_3(x)-\mu\nu K^2x^{-2\alpha}G_2(x)G_4(x)\right),\\
\Sigma_{22}(x)&=\frac{L}{\gamma}\left(G_3^2(x)-\mu\nu^2 K^2x^{-2\alpha}G_4^2(x)\right),
\end{split}
\end{equation*}
and
\begin{align*}
G_1(x)&={}_2 F_1\left(\alpha+\beta+t+1,-t; \alpha+1;x\right),\quad G_2(x)={}_2 F_1\left(\beta+t+1,-\alpha-t; 1-\alpha;x\right),\\
G_3(x)&={}_2 F_1\left(\alpha+\beta+t,1-t; \alpha+1;x\right),\quad\;\;\; G_4(x)={}_2 F_1\left(\beta+t,1-t-\alpha; 1-\alpha;x\right),\\
\mu&=\frac{\sin(\pi t)\sin(\pi(\beta+t))}{\sin(\pi(\alpha+\beta+t))\sin(\pi(\alpha+t))},\quad \nu=\frac{(\alpha+t)(\alpha+\beta+t)}{t(\beta+t)},\\
K&=-\frac{\Gamma(\alpha)\Gamma(\alpha+1)\Gamma(t+1)\Gamma(-\alpha-\beta-t)\sin(\pi\alpha)\sin(\pi(\alpha+\beta+t))}{\pi\Gamma(\alpha+t+1)\Gamma(-\beta-t)\sin(\pi(\beta+t))},\\
L&=\frac{t(\beta+t)\sin(\pi\alpha)}{\pi\sqrt{a_0}(\alpha+\beta+2t)\alpha(\mu-1)K},\quad\gamma=\frac{(\alpha+t)(\alpha+\beta+t)}{\sqrt{a_0}(\alpha+\beta+2t-1)(\alpha+\beta+2t)}.
\end{align*}
where $a_0$ can be obtained from \eqref{TTRRcn} and $_2F_1$ denotes the standard Gauss hypergeometric function.

Another way of representing the orthonormality of $(P_n^\eta)_{n\in\mathbb{Z}}, \eta=1,2,$ is using the theory of matrix-valued orthogonal polynomials, since a doubly infinite Jacobi matrix like in \eqref{JZ} can be viewed as a semi-infinite $2\times2$ block Jacobi matrix (see \cite{Ber}). Indeed, after the new labeling of the indices $n$
\begin{equation}\label{relab}
\{0,1,2,\ldots\}\to\{0,2,4,\ldots\},\quad\mbox{and}\quad\{-1,-2,-3,\ldots\}\to\{1,3,5,\ldots\},
\end{equation}
all the information of $J$ can be collected in a semi-infinite block tridiagonal matrix $ \bm J $ of the form
\begin{equation*}
\bm J=\begin{pmatrix}
B_1&A_1&&&\\
A_1&B_2&A_2&&\\
&A_2&B_3&A_3&\\
&&\ddots&\ddots&\ddots
\end{pmatrix},
\end{equation*}
where
\begin{align*}
B_1&=\begin{pmatrix} b_1 & \sqrt{a_0}\\\sqrt{a_{0}} & b_{0}\end{pmatrix},\quad B_{n+1}=\begin{pmatrix} b_{n+1} & 0\\ 0 & b_{-n}\end{pmatrix},\quad n\geq1,\quad A_{n+1}=\begin{pmatrix} \sqrt{a_{n+1}} & 0\\ 0 & \sqrt{a_{-n-1}}\end{pmatrix},\quad n\geq0.
\end{align*}
If we define the matrix-valued polynomials
\begin{equation}\label{2QMMP}
\bm P_n(x)=\begin{pmatrix} P_n^1(x) & P_n^2(x) \\ P_{-n-1}^1(x) & P_{-n-1}^2(x)\end{pmatrix},\quad n\geq0,
\end{equation}
then we have
\begin{align*}
x\bm P_0(x)&=A_1\bm P_{1}(x)+B_1\bm P_0(x),\quad \bm P_0(x)=I_{2\times2},\\
x\bm P_n(x)&=A_{n+1}\bm P_{n+1}(x)+B_{n+1}\bm P_n(x)+A_n\bm P_{n-1}(x),\quad n\geq1,
\end{align*} 
where $I_{2\times2}$ denotes the $2\times2$ identity matrix. The matrix orthonormality is defined in terms of the (matrix-valued) inner product
\begin{equation*}\label{ortoZ3P}
\int_{0}^{1}\bm P_n(x)\bm W(x)\bm P_m^*(x)dx=I_{2\times2}\delta_{nm},
\end{equation*}
where $M^*$ denotes the Hermitian transpose of a matrix $M$.
\begin{remark}
Observe that in \cite{GH} the authors use a different representation of the matrix-valued orthonormal polynomials $\bm P_n$ in \eqref{2QMMP}. It is easy to see that both notations are connected by a conjugation of the $\sigma_1$ Pauli matrix (of size $2\times2$). Thus the spectral matrix in \eqref{WW} has been modified according to our notation.
\end{remark}
\begin{remark}\label{rem1}
Observe that if we denote by $\tilde a_n,\tilde b_n, n\geq1,$ the coefficients of the three-term recurrence relation for the classical Jacobi orthonormal polynomials on the interval $[0,1]$, we have that $a_n=\tilde a_{n+t}$ and $b_n=\tilde b_{n+t}$, for $n\in\mathbb{Z}$, something that it was already pointed out in \cite{GH}.
\end{remark}

\begin{remark}\label{rembis}
In \cite{GH0}, a complete solution of the discrete-continuous version of the following bispectral problem was given: describe all families of \emph{functions} $f_n(z), n\in\mathbb{Z}, z\in\mathbb{C}$, that satisfy
$$
(Jf)_n(z)=zf_n(z),\quad\mbox{and}\quad Bf_n(z)=\lambda_nf_n(z),\quad\mbox{for all $z$ and $n$,}
$$
where $J$ is a doubly infinite Jacobi matrix like in \eqref{JZ} and $B$ is a second-order differential operator with coefficients independent of $n$.
In the case of associated Jacobi polynomials, $J$ is the Jacobi matrix \eqref{JZ} with coefficients \eqref{TTRRcn} and the functions $f_n(z)$ can be given in terms of an arbitrary solution of Gauss' hypergeometric equation (see (2.1)--(2.6) of \cite{GH}). The corresponding second-order differential operator $B$ and eigenvalue $\lambda_n$ are given by
$$
B=z(1-z)\frac{d^2}{dz^2}+(\alpha+1+(\alpha+\beta+2)z)\frac{d}{dz},\quad \lambda_n=-(n+t)(n+t+\alpha+\beta+1).
$$
We will get back to the notion of bispectrality at the end of Section \ref{sec3}, where we will show that the corresponding matrix-valued orthogonal polynomials $\bm P_n$ in \eqref{2QMMP} also have the bispectral property.
\end{remark}

\section{Stochastic associated Jacobi matrix}\label{sec3}

We are now interested in turning the doubly infinite Jacobi matrix $J$ in \eqref{JZ} into a stochastic matrix $P$. For that we will follow Remark \ref{rem1} and \cite[Section 5]{GdI3} (see also \cite{G4}). For $n\in\mathbb{Z}$ let
\begin{equation}\label{coefTTRR}
\begin{split}
p_n&=\frac{(n+t+\beta+1)(n+t+\alpha+\beta+1)}{(2n+2t+\alpha+\beta+1)(2n+2t+\alpha+\beta+2)},\\
r_n&=\frac{(n+t+\beta+1)(n+t+1)}{(2n+2t+\alpha+\beta+1)(2n+2t+\alpha+\beta+2)}+\frac{(n+t+\alpha)(n+t+\alpha+\beta)}{(2n+2t+\alpha+\beta)(2n+2t+\alpha+\beta+1)},\\
q_n&=\frac{(n+t)(n+t+\alpha)}{(2n+2t+\alpha+\beta)(2n+2t+\alpha+\beta+1)}.
\end{split}
\end{equation}
Observe that $p_n+r_n+q_n=1, n\in\mathbb{Z}$. If we denote by $\tilde p_n,\tilde r_n,\tilde q_n$ the coefficients of the three-term recurrence relation for the classical Jacobi polynomials on $[0,1]$ such that the corresponding Jacobi matrix is a stochastic matrix (in \cite{GdI3} they are denoted by $a_n,b_n,c_n$) then we have $p_n=\tilde p_{n+t}, r_n=\tilde r_{n+t}, q_n=\tilde q_{n+t}, n\in\mathbb{Z}$. The corresponding doubly infinite Jacobi matrix can be written in the following form
\begin{equation}\label{QZ}
P=
\left(
\begin{array}{cccc|cccc}
\ddots&\ddots&\ddots&&&\\
&q_{-2}&r_{-2}&p_{-2}&&&\\
&&q_{-1}&r_{-1}&p_{-1}&&&\\
\hline
&&&q_0&r_0&p_0&&\\
&&&&q_1&r_1&p_1&\\
&&&&&\ddots&\ddots&\ddots
\end{array}
\right).
\end{equation}
Let us define the so-called \emph{potential coefficients} as
\begin{equation*}
\pi_0=1,\quad \pi_n=\frac{p_0p_1\cdots p_{n-1}}{q_1q_2\cdots q_n},\quad \pi_{-n}=\frac{q_0q_{-1}\cdots q_{-n+1}}{p_{-1}p_{-2}\cdots p_{-n}},\quad n\geq1.
\end{equation*}
It is possible to see from \eqref{coefTTRR} that $\pi=(\pi_n)_{n\in\ZZ}$ is explicitly given by
\begin{equation}\label{potcoeff}
\pi_n=\frac{(\alpha+\beta+t+1)_n(-t)_{-n}(2n+2t+\alpha+\beta+1)}{(\alpha+t+1)_n(-\beta-t)_{-n}(2t+\alpha+\beta+1)},\quad n\in\mathbb{Z},
\end{equation}
where $(a)_0=1, (a)_n=a(a+1)\cdots(a+n-1),n\geq1,$ denotes the Pochhammer symbol and $(a)_{-n}=1/(a-n)_n,n\geq1$. If we write $\Pi=\mbox{diag}\left(\cdots,\sqrt{\pi_{-1}},\sqrt{\pi_{0}},\sqrt{\pi_1},\cdots\right)$ the relation between the Jacobi matrices $J$ and $P$ given by \eqref{JZ} and \eqref{QZ}, respectively, is $J\Pi=\Pi P$. In the following result we will see under what conditions we can ensure that $p_n,r_n,q_n>0$ for all $n\in\mathbb{Z}$. Recall from Section \ref{sec2} that $\alpha$ and $\beta$ are inside the square $-1<\alpha,\beta<1$.
\begin{proposition}\label{prop1}
If $0\leq\beta<1$ then it is not possible to find $t$ such that $p_n,r_n,q_n>0$ for all $n\in\mathbb{Z}$. Therefore assume that $-1<\alpha<1$ and $-1<\beta<0$. Then the coefficients $p_n,r_n,q_n$ defined by \eqref{coefTTRR} are all positive if we choose $t$ according to one of the following 8 regions (see Figure \ref{fig1}):
\begin{figure}[h]
\includegraphics[scale = 0.3]{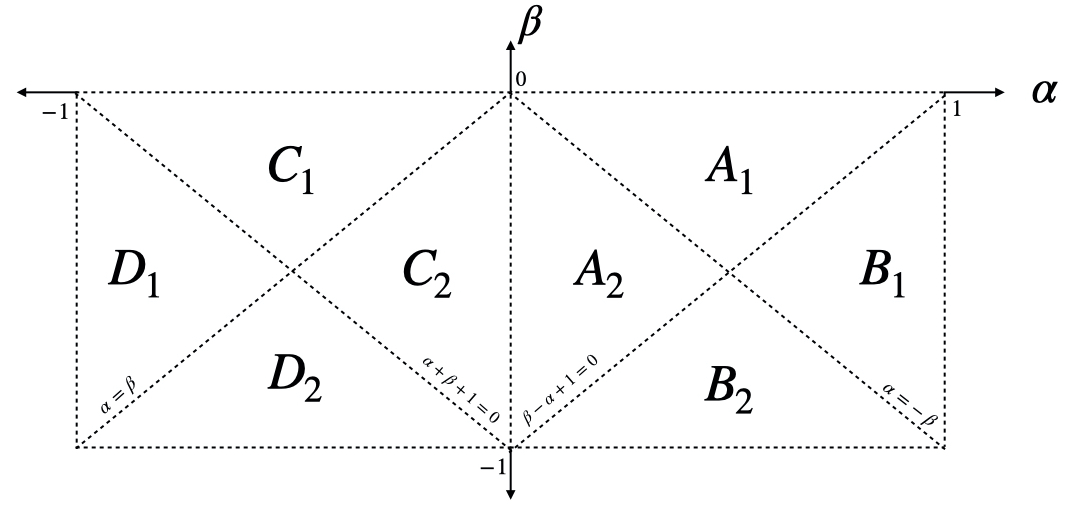}
\caption{Representation of the regions in Proposition \ref{prop1}.}
\label{fig1}
\end{figure}
\begin{description}
\item[$\bm A_1$] $\{\beta-\alpha+1>0,\alpha>-\beta,\beta<0\}$. Then $t$ must be chosen in the following real set:
\begin{equation}\label{setA1}
t\in\bigcup_{n\in\mathbb{Z}}(n,n-\beta)\cup(n-\alpha,n-\alpha-\beta).
\end{equation}
\item[$\bm A_2$] $\{\beta-\alpha+1>0,\alpha<-\beta,\alpha>0\}$. Then $t$ must be chosen in the following real set:
$$
t\in\bigcup_{n\in\mathbb{Z}}(n-\alpha,n)\cup(n-\alpha-\beta,n-\beta).
$$
\item[$\bm B_1$] $\{\beta-\alpha+1<0,\alpha>-\beta,\alpha<1\}$. Then $t$ must be chosen in the following real set:
$$
t\in\bigcup_{n\in\mathbb{Z}}(n,n-\alpha+1)\cup(n-\beta,n-\alpha-\beta+1).
$$
\item[$\bm B_2$] $\{\beta-\alpha+1<0,\alpha<-\beta,\beta>-1\}$. Then $t$ must be chosen in the following real set:
$$
t\in\bigcup_{n\in\mathbb{Z}}(n-\beta-1,n)\cup(n-\alpha-\beta,n-\alpha+1).
$$
\item[$\bm C_1$] $\{\beta+\alpha+1>0,\alpha>\beta,\alpha<0\}$. Then $t$ must be chosen in the following real set:
$$
t\in\bigcup_{n\in\mathbb{Z}}(n,n-\alpha)\cup(n-\beta,n-\alpha-\beta).
$$
\item[$\bm C_2$] $\{\beta+\alpha+1>0,\alpha<\beta,\beta<0\}$. Then $t$ must be chosen in the following real set:
$$
t\in\bigcup_{n\in\mathbb{Z}}(n,n-\beta)\cup(n-\alpha,n-\alpha-\beta).
$$
\item[$\bm D_1$] $\{\beta+\alpha+1<0,\alpha>\beta,\beta>-1\}$. Then $t$ must be chosen in the following real set:
$$
t\in\bigcup_{n\in\mathbb{Z}}(n-\beta-1,n)\cup(n-\alpha-\beta,n-\alpha+1).
$$
\item[$\bm D_2$] $\{\beta+\alpha+1<0,\alpha<\beta,\alpha>-1\}$. Then $t$ must be chosen in the following real set:
$$
t\in\bigcup_{n\in\mathbb{Z}}(n-\alpha-1,n)\cup(n-\alpha-\beta,n-\beta+1).
$$
\end{description}
\end{proposition}
\begin{proof}
Let us write the coefficients $p_n,r_n,q_n$ in \eqref{coefTTRR} in the following more convenient way
\begin{equation}\label{prqprop}
p_n=x_ns_{n+1},\quad r_n=x_nt_{n+1}+y_ns_n,\quad q_n=t_ny_n,\quad n\in\ZZ,
\end{equation}
where for $n\in\ZZ$
\begin{equation}\label{xysrprop}
x_n=\frac{n+t+\beta+1}{2n+2t+\alpha+\beta+1},\; y_n=\frac{n+t+\alpha}{2n+2t+\alpha+\beta+1},\; s_n=\frac{n+t+\alpha+\beta}{2n+2t+\alpha+\beta},\;  t_n=\frac{n+t}{2n+2t+\alpha+\beta}.
\end{equation}
This representation is motivated by the UL factorization of the doubly infinite matrix $P$ \eqref{QZ}. For reasons that we will explain in the next Section \ref{sec41}, we want the sequences in \eqref{xysrprop} to be probabilities for all $n\in\ZZ$, since we are interested in an urn model associated with the transition probability matrix $P$ \eqref{QZ} (see Section \ref{sec6}).

Let us now prove the first part of the proposition. For simplicity we restrict ourselves to the region $R=\{\alpha>\beta\geq0,\beta<1-\alpha\}$. The rest of cases are similar. Under the conditions on the region $R$ we have that the values of the zeros of each of the factors in \eqref{xysrprop} (for $n+t$ as a variable) are located as in Figure \ref{figone} (not to scale). 
\begin{figure}[h]
\includegraphics[scale = 0.35]{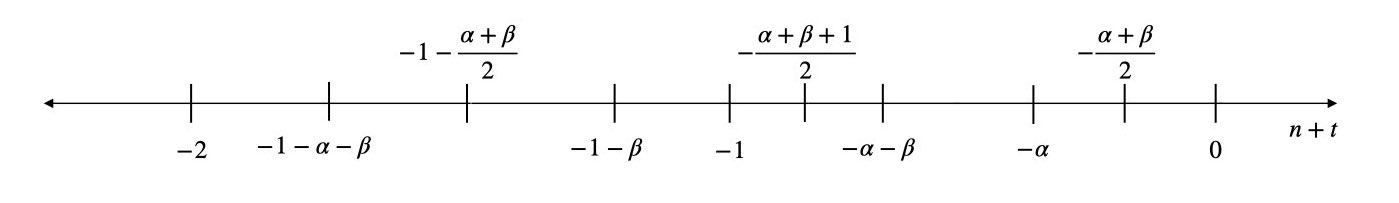}
\caption{Position of different values of the zeros for which $n+t$ vanishes in each of the factors of \eqref{coefTTRR} for the first part of the Proposition \ref{prop1}.}
\label{figone}
\end{figure}
It is enough to see what happens in the interval $[-1,0]$, since $n+t$ can be moved to any other interval of this size for some integer $n$. In the following analysis we will be focusing on the sign of $x_n,y_n,s_n,t_n$ in \eqref{xysrprop} and consequently the sign of $p_n,r_n,q_n$ in \eqref{prqprop}. We have 4 possibilities according to the position of $n+t$ (see Figure \ref{figone}):
\begin{enumerate}
\item If $-1<n+t<-\frac{\alpha+\beta+1}{2}$, then we have $x_n=\frac{(+)}{(-)}$ and $s_{n+1}=\frac{(+)}{(+)}$. Therefore $p_n=x_ns_{n+1}<0$.
\item If $-\frac{\alpha+\beta+1}{2}<n+t<-\alpha$, then we have $t_n=\frac{(-)}{(-)}$ and $y_{n}=\frac{(-)}{(+)}$. Therefore $q_n=t_ny_{n}<0$.
\item If $-\alpha<n+t<-\frac{\alpha+\beta}{2}$, then we have $y_n=\frac{(+)}{(+)}$ and $s_{n}=\frac{(+)}{(-)}$. Therefore $y_ns_{n}<0$ and the second summand in $r_n$ of \eqref{prqprop} is negative.
\item If $-\frac{\alpha+\beta}{2}<n+t<0$, then we have $t_n=\frac{(-)}{(+)}$ and $y_{n}=\frac{(+)}{(+)}$. Therefore $q_n=t_ny_{n}<0$.
\end{enumerate}
In all the 4 cases below it is not possible to find $t$ such that $p_n,r_n,q_n>0$ for all $n\in\ZZ$. Finally, if $n+t>0$ then all $x_n,y_n,s_n,t_n$ in \eqref{xysrprop} are positive and $p_n,r_n,q_n>0$ for all $n\in\ZZ$. But then we have to choose $t$ such that $t>-n$. As $|n|\to\infty$ it will not be possible to find a finite $t$ such that $p_n,r_n,q_n>0$ for all $n\in\ZZ$. The same can be applied if $n+t<-1-\alpha-\beta$, in which case all $x_n,y_n,s_n,t_n$ in \eqref{xysrprop} are negative and $p_n,r_n,q_n>0$ for all $n\in\ZZ$. The rectangle $-1<\alpha<1, 0\leq\beta<1$ can be divided in a similar way as in Figure \ref{fig1}. The region $R$ is just one of these triangles. The proof for the rest of the regions is similar, only changing the position of the values of the zeros in Figure \ref{figone}.

For the second part of the proposition we will focus on the region $\bm A_1$. The rest of cases are similar. As before, we have that the values of the zeros of each of the factors in \eqref{xysrprop} (for $n+t$ as a variable) are located as in Figure \ref{figtwo} (not to scale),
\begin{figure}[h]
\includegraphics[scale = 0.35]{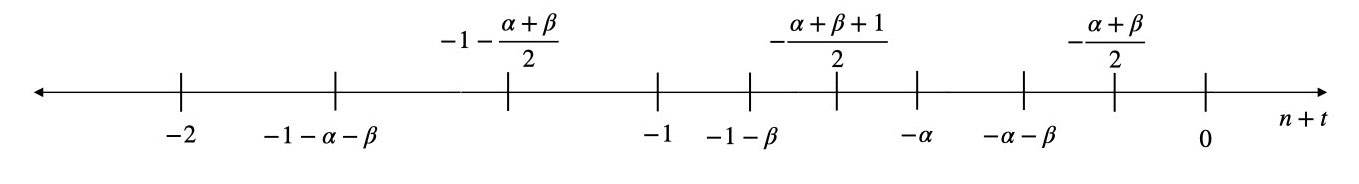}
\caption{Position of different values of the zeros for which $n+t$ vanishes in each of the factors of \eqref{coefTTRR} for the second part of the Proposition \ref{prop1}.}
\label{figtwo}
\end{figure}
so it is enough to see what happens in the interval $[-1,0]$. We have now 6 possibilities according to the position of $n+t$:
\begin{enumerate}
\item If $-1<n+t<-1-\beta$, then we have $x_n=\frac{(-)}{(-)}$, $s_{n+1}=\frac{(+)}{(+)}$, $t_{n+1}=\frac{(+)}{(+)}$, $y_n=\frac{(-)}{(-)}$, $s_{n}=\frac{(-)}{(-)}$ and $t_n=\frac{(-)}{(-)}$. Therefore $p_n,r_n,q_n>0$ for all $n\in\ZZ$.
\item If $-1-\beta<n+t<-\frac{\alpha+\beta+1}{2}$, then we have $x_n=\frac{(+)}{(-)}$ and $s_{n+1}=\frac{(+)}{(+)}$. Therefore $p_n=x_ns_{n+1}<0$.
\item If $-\frac{\alpha+\beta+1}{2}<n+t<-\alpha$, then we have $t_n=\frac{(-)}{(-)}$ and $y_{n}=\frac{(-)}{(+)}$. Therefore $q_n=t_ny_{n}<0$.
\item If $-\alpha<n+t<-\alpha-\beta$, then we have $x_n=\frac{(+)}{(+)}$, $s_{n+1}=\frac{(+)}{(+)}$, $t_{n+1}=\frac{(+)}{(+)}$, $y_n=\frac{(+)}{(+)}$, $s_{n}=\frac{(-)}{(-)}$ and $t_n=\frac{(-)}{(-)}$. Therefore $p_n,r_n,q_n>0$ for all $n\in\ZZ$.
\item If $-\alpha-\beta<n+t<-\frac{\alpha+\beta}{2}$, then we have $y_n=\frac{(-)}{(-)}$ and $s_{n}=\frac{(+)}{(-)}$. Therefore $y_ns_{n}<0$ and the second summand in $r_n$ of \eqref{prqprop} is negative.
\item If $-\frac{\alpha+\beta}{2}<n+t<0$, then we have $t_n=\frac{(-)}{(+)}$ and $y_{n}=\frac{(+)}{(+)}$. Therefore $q_n=t_ny_{n}<0$.
\end{enumerate}
The previous cases (1) and (4) correspond to what we wanted to prove in \eqref{setA1}. The proof for the rest of the regions is similar, only changing the position of the values in Figure \ref{figtwo}.
\end{proof}

Under the conditions of the previous proposition we can ensure that the doubly infinite Jacobi matrix $P$ in \eqref{QZ} is \emph{stochastic}. Therefore it can interpreted as the one-step transition probability matrix of a nontrivial bilateral birth-death chain $\{Z_t : t=0,1,\ldots\}$ on $\mathbb{Z}$ with diagram given by
\begin{center}
\includegraphics[scale=.30]{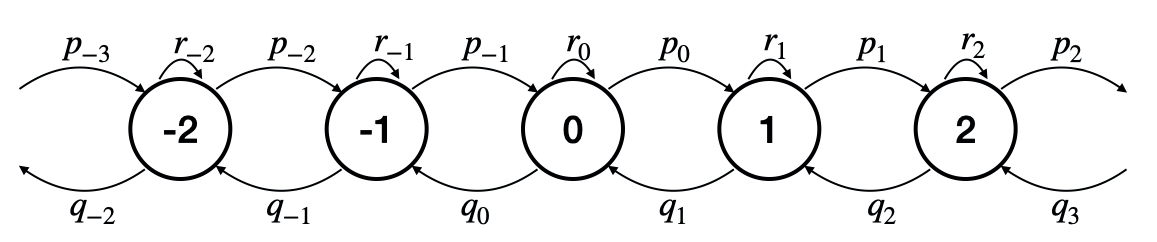}
\end{center}
depending on three parameters $\alpha,\beta,t$. As in Section \ref{sec2}, we can define two sets of linearly independent polynomials $(Q_n^{\eta}(x))_{n\in\ZZ}, \eta=1,2,$ by the three-term recurrence relation with initial conditions
\begin{align*}
\nonumber Q_0^1(x)&=1,\quad Q_{0}^2(x)=0,\\
\label{TTRRZ}Q_{-1}^1(x)&=0,\quad Q_{-1}^2(x)=1,\\
\nonumber xQ_n^{\eta}(x)&=p_nQ_{n+1}^{\eta}(x)+r_nQ_n^{\eta}(x)+q_nQ_{n-1}^{\eta}(x),\quad n\in\ZZ,\quad\eta=1,2.
\end{align*}
A closed form of some of these polynomials can be computed using \cite[Theorem 1]{Wi} in terms of $_4F_3$ hypergeometric functions (see also \cite[Section 4.4]{AIB}). These polynomials are orthogonal in the following sense
\begin{equation*}\label{ortoZ2}
\int_0^1\begin{pmatrix} Q_i^1(x),&\hspace{-0.3cm}Q_i^2(x)\end{pmatrix}\bm \Psi(x)\begin{pmatrix} Q_j^1(x)\\Q_j^2(x)\end{pmatrix}dx=\frac{\delta_{ij}}{\pi_j},\quad i,j\in\ZZ,
\end{equation*}
where $(\pi_n)_{n\in\ZZ}$ are defined by \eqref{potcoeff} and (see \eqref{WW})
\begin{equation}\label{psiw}
\bm\Psi(x)=\begin{pmatrix} 1&0\\0&1/\sqrt{\pi_{-1}}\end{pmatrix}\bm W(x)\begin{pmatrix} 1&0\\0&1/\sqrt{\pi_{-1}}\end{pmatrix}=x^\alpha(1-x)^\beta
\begin{pmatrix} \Sigma_{11}(x)&\D\frac{1}{\sqrt{\pi_{-1}}}\Sigma_{12}(x)\\\D\frac{1}{\sqrt{\pi_{-1}}}\Sigma_{12}(x)&\D\frac{1}{\pi_{-1}}\Sigma_{22}(x)\end{pmatrix},\quad x\in[0,1].
\end{equation}
With this information we can compute the $n$-step transition probabilities of the bilateral birth-death chain $\{Z_t : t=0,1,\ldots\}$, given by the Karlin-McGregor integral representation formula (see \cite{KMc6})
\begin{equation*}\label{KmcG1}
P_{ij}^{(n)}\doteq\mathbb{P}(Z_n=j \; | Z_0=i)=\pi_j\int_{0}^{1}x^n\begin{pmatrix} Q_i^1(x),&\hspace{-0.3cm}Q_i^2(x)\end{pmatrix}\bm\Psi(x)\begin{pmatrix}Q_j^1(x)\\Q_j^2(x)\end{pmatrix}dx,\quad i,j\in\ZZ.
\end{equation*}
Again, after the relabeling \eqref{relab}, all the information of $P$ can be collected in a semi-infinite block tridiagonal matrix $\bm P$ with blocks of size $ 2 \times2 $, given by
\begin{equation}\label{Pmblock}
\bm P=\begin{pmatrix}
E_0&D_0&&&\\
F_1&E_1&D_1&&\\
&F_2&E_2&D_2&\\
&&\ddots&\ddots&\ddots
\end{pmatrix},
\end{equation}
where
\begin{align*}
E_0&=\begin{pmatrix} r_0 & q_0\\p_{-1} & r_{-1}\end{pmatrix},\quad E_n=\begin{pmatrix} r_n & 0\\ 0 & r_{-n-1}\end{pmatrix},\quad n\geq1,\\
D_n&=\begin{pmatrix} p_n & 0\\ 0 & q_{-n-1}\end{pmatrix},\quad n\geq0,\quad F_n=\begin{pmatrix} q_n & 0\\ 0 & p_{-n-1}\end{pmatrix},\quad n\geq1.
\end{align*}
The birth-death chain generated by $\bm P $ can be interpreted as a process that takes values in the two-dimensional state space $\ZZ_{\geq0}\times \{1,2 \}$, where the first component is usually called the \emph{level} while the second component is called the \emph{phase}. These type of processes are also known as discrete-time \emph {quasi-birth-and-death processes}. In general these processes allow transitions between all adjacent states levels and all phases (see \cite{LaR, Neu} for general references). As before, if we define the matrix-valued polynomials
\begin{equation}\label{2QMM}
\bm Q_n(x)=\begin{pmatrix} Q_n^1(x) & Q_n^2(x) \\ Q_{-n-1}^1(x) & Q_{-n-1}^2(x)\end{pmatrix},\quad n\geq0,
\end{equation}
then we have
\begin{equation}\label{TTRRQ}
\begin{split}
x\bm Q_0(x)&=D_0\bm Q_{1}(x)+E_0\bm Q_0(x),\quad \bm Q_0(x)=I_{2\times2},\\
x\bm Q_n(x)&=D_n\bm Q_{n+1}(x)+E_n\bm Q_n(x)+F_n\bm Q_{n-1}(x),\quad n\geq1,
\end{split} 
\end{equation}
and
\begin{equation*}\label{ortoZ3}
\int_{0}^{1}\bm Q_n(x)\bm\Psi(x)\bm Q_m^*(x)dx=\begin{pmatrix} 1/\pi_n & 0 \\ 0 & 1/\pi_{-n-1}\end{pmatrix}\delta_{nm}.
\end{equation*}
The relation between the matrix-valued polynomials $\bm Q_n(x)$ in \eqref{2QMM} and $\bm P_n(x)$ in \eqref{2QMMP} is given by
$$
\bm P_n(x)=\begin{pmatrix} \sqrt{\pi_n} & 0 \\ 0 & \sqrt{\pi_{-n-1}} \end{pmatrix}\bm Q_n(x)\begin{pmatrix} 1 & 0 \\ 0 & 1/\sqrt{\pi_{-1}} \end{pmatrix},\quad n\geq0.
$$
In this case we have (see \cite{DRSZ, G2}) the Karlin-McGregor integral representation formula where the $2\times2$ block entry $(i,j)$ is given by
\begin{equation}\label{KMcG1}
\bm P_{ij}^{(n)}=\left(\int_{0}^{1}x^n\bm Q_i(x)\bm\Psi(x)\bm Q_j^*(x)dx\right)\begin{pmatrix} \pi_j & 0 \\ 0 & \pi_{-j-1}\end{pmatrix},\quad i,j\in\ZZ_{\geq0}.
\end{equation}

Finally, let us make some remarks about the recurrence and the invariant measure of this family of discrete-time bilateral birth-death chains. From Corollaries 4.1 and 4.2 of \cite{DRSZ} (see also \cite[Remark 3.1]{dIJ1}), it is possible to derive that a birth-death chain on $\mathbb{Z}$ is \emph{recurrent} if and only if at least one entry of the integral $\int_{0}^1\bm\Psi(x)/(1-x)dx$ is divergent, and it is positive recurrent if and only if one of the entries of $\bm\Psi(x)$ has a jump at the point 1. Taking a look at the spectral matrix $\bm\Psi$ (see \eqref{psiw} and \eqref{WW}) it is possible to see that, under the conditions of  Proposition \ref{prop1}, the functions inside the matrix are bounded at the point $x=1$. Therefore the divergence of the integral $\int_{0}^1\bm\Psi(x)/(1-x)dx$ only depends on the part $(1-x)^\beta$ from the scalar measure. Under the conditions of Proposition \ref{prop1}, i.e. $-1<\beta<0$, we always have that all entries of $\int_{0}^1\bm\Psi(x)/(1-x)dx$ are divergent. Therefore the bilateral birth-death chain is always recurrent. Since the spectral measure has only an absolutely continuous part, there are no jumps at the point 1. Therefore the birth-death chain is always \emph{null recurrent}. As for the invariant measure, this is directly given by the vector $\bm\pi=(\cdots, \pi_{-1},\pi_{0},\pi_{1},\cdots)$ where $(\pi_n)_{n\in\ZZ}$ is given by \eqref{potcoeff}.

\begin{remark}\label{rem32}
Instead of \eqref{coefTTRR}, which lead us to a discrete-time bilateral birth-death chain, we could have considered the polynomials $(Q_n^{\eta}(x))_{n\in\ZZ}, \eta=1,2,$ generated by the three-term recurrence coefficients $p_n, q_n$ and $r_n=-p_n-q_n, n\in\mathbb{Z}$. Then we will get a continuous-time bilateral birth-death process and similar spectral analysis can be performed (see \cite{dI3}). This was done in \cite{IM} for the case where the state space is in $\mathbb{Z}_{\geq0}$. In this paper we focus on the discrete-time bilateral birth-death chain in order to give a probabilistic interpretation in terms of urn models (see Section \ref{sec6}).
\end{remark}

\begin{remark}
Let us return to the subject of \emph{bispectrality} pointed out in Remark \ref{rembis}. The matrix-valued orthogonal polynomials $(\bm Q_n)_{n\geq0}$ \eqref{2QMM} are \emph{bispectral}. Not only they are eigenfuntions of a block tridiagonal Jacobi operator $\bm P$ like in \eqref{Pmblock} (see also \eqref{TTRRQ}), but  they are also eigenfunctions of the following matrix-valued second-order differential operator
\begin{equation}\label{sode}
\bm B=x(1-x)\frac{d^2}{dx^2}+(C-xU)\frac{d}{dx},
\end{equation}
i.e. $\bm Q_n(x)\bm B=\Lambda_n\bm Q_n(x)$, where
\begin{equation}\label{CUVL}
\begin{split}
C&=\begin{pmatrix}
\alpha+1+\D\frac{2t(\beta+t)}{\alpha+\beta+2t} & 2t-\D\frac{2t(\beta+t)}{\alpha+\beta+2t}\\
-2(\beta+t)+\D\frac{2t(\alpha+t)}{\alpha+\beta+2t} & 1-\alpha-\D\frac{2t(\beta+t)}{\alpha+\beta+2t}
\end{pmatrix},\quad U=\begin{pmatrix}
\alpha+\beta+2t+2& 0\\
0& -\alpha-\beta-2t+2
\end{pmatrix},\\
V&=\begin{pmatrix}
-\alpha-\beta-2t & 0\\
0&0
\end{pmatrix},\quad \Lambda_n=\begin{pmatrix}
-(n+1)(n+\alpha+\beta+2t)& 0\\
0& -n(n-\alpha-\beta-2t+1)
\end{pmatrix},\quad n\geq0.
\end{split}
\end{equation}
Observe that the coefficients of  $\bm B$ (independent of $n$) are multiplied on the right while the eigenvalue $\Lambda_n$ is multiplied on the left. This is consistent with the theory of matrix-valued orthogonal polynomials satisfying second-order differential equations initiated by A.J. Dur\'an, F.A. Gr\"unbaum, I. Pacharoni and J.A. Tirao (see \cite{DG1,DG2,GPT1,GPT2}). In terms of the two linearly independent families of polynomials $(Q_n^{\eta}(x))_{n\in\ZZ}, \eta=1,2,$ we have two coupled second-order differential equations of the form
\begin{align*}
x(1-&x)\left(Q_n^{\eta}(x)\right)''+\left(1+\varepsilon\left(\alpha+\frac{2t(\beta+t)}{\alpha+\beta+2t}\right)-x(2+\varepsilon(\alpha+\beta+2t))\right)\left(Q_n^{\eta}(x)\right)'\\
&+\left(-\beta(1+\varepsilon)+2\varepsilon t\left(-1+\frac{\beta+t}{\alpha+\beta+2t}\right)\right)\left(Q_n^{\eta+\varepsilon}(x)\right)'-\frac{1}{2}(1+\varepsilon)(\alpha+\beta+2t)Q_n^{\eta}(x)=\lambda_n^{\pm}Q_n^{\eta}(x),\quad n\in\mathbb{Z},
\end{align*}
where
$$
\varepsilon=\begin{cases}1,&\mbox{if} \quad\eta=1\\-1,&\mbox{if}\quad\eta=2\end{cases},\quad \lambda_n^{\pm}=\begin{cases}-(n+1)(n+\alpha+\beta+2t),&\mbox{if} \quad n\geq0\\-n(n-\alpha-\beta-2t+1),&\mbox{if}\quad n<0\end{cases}.
$$
Finally, let us point out that J. Wimp \cite{Wi} found a fourth-order differential equation with coefficients \emph{depending on $n$} for the family of polynomials $(Q_n^1(x))_{n\geq0}$.

\end{remark}

\section{Stochastic factorizations of the associated Jacobi matrix}\label{sec4}

Our goal now is to find an application of the bilateral birth-death chain with transition probability matrix $P$ \eqref{QZ} in terms of urn models (see Section \ref{sec6} below). To that end, we will follow the ideas of \cite{GdI3} (see also \cite{dIJ1,dIJ2}) and divide the urn model associated with $P$ into two different and simpler urn experiments, and combine them to obtain a simpler description of the original urn model. For that we will study stochastic UL and LU decompositions of the transition probability matrix $P$ \eqref{QZ}. These factorizations typically come with a free parameter, but we will see that in the case of associated Jacobi polynomials, the factorizations, if possible, are unique.

\subsection{Stochastic UL factorization}\label{sec41}

Let us perform a UL factorization of $P$ \eqref{QZ} in the following way
\begin{equation}\label{QZUL}
P=\left(\begin{array}{cc|cccc}
\ddots&\ddots&&\\
0&y_{-1}&x_{-1}&&\\
\hline
&0&y_0&x_0&\\
&&0&y_1&x_1\\
&&&&\ddots&\ddots
\end{array}
\right)\left(\begin{array}{ccc|ccc}
\ddots&\ddots&&&\\
&t_{-1}&s_{-1}&0&\\
\hline
&&t_0&s_0&0\\
&&&t_1&s_1&0\\
&&&&\ddots&\ddots
\end{array}
\right)=P_UP_L,
\end{equation}
where $P_U$ and $P_L$ are also stochastic matrices. This means that all entries of $P_U$ and $P_L$ are nonnegative and 
\begin{equation*}\label{stpul}
x_n+y_n=1,\quad s_n+t_n=1,\quad n\in\ZZ.
\end{equation*}
A direct computation shows that
\begin{align}
\nonumber p_n&=x_ns_{n+1},\\
\label{ULd}r_n&=x_nt_{n+1}+y_ns_n,\quad n\in\ZZ,\\
\nonumber q_n&=y_nt_n.
\end{align}
This UL factorization typically depends on one free parameter $y_0$. If we want the factors $P_U$ and $P_L$ to be also stochastic matrices, then we have to apply \cite[Theorem 2.1]{dIJ1}, which says (under certain conditions) that the factorization is stochastic if and only if we choose $y_0$ in the following range
$$
H'\leq y_0\leq H,
$$
where $H$ and $H'$ are the following continued fractions generated by alternatively choosing $p_n$ and $q_n$ in different directions, i.e.
\begin{equation}\label{cfUL}
H=1-\cfrac{p_0}{1-\cfrac{q_1}{1-\cfrac{p_1}{1-\cfrac{q_2}{1-\cdots}}}},\quad H'=\cfrac{q_0}{1-\cfrac{p_{-1}}{1-\cfrac{q_{-1}}{1-\cfrac{p_{-2}}{1-\cdots}}}}.
\end{equation}

\begin{theorem}\label{thm1}
Assume that $\alpha>0$. Then we have that
\begin{equation}\label{HHp}
H=H'=\frac{\alpha+t}{\alpha+\beta+2t+1}.
\end{equation}
Therefore there exists only one value of the parameter $y_0$ ($y_0=H$) such that we obtain a stochastic UL factorization of the form \eqref{QZUL} and the coefficients of each of the factors $P_U$ and $P_L$ are given by
\begin{equation}\label{xysr1}
\begin{split}
y_n&=\frac{n+t+\alpha}{2n+2t+\alpha+\beta+1},\quad x_n=\frac{n+t+\beta+1}{2n+2t+\alpha+\beta+1},\\
s_n&=\frac{n+t+\alpha+\beta}{2n+2t+\alpha+\beta},\qquad\;\;\; t_n=\frac{n+t}{2n+2t+\alpha+\beta},
\end{split}\qquad n\in\ZZ.
\end{equation}
\end{theorem}

\begin{proof}
We will follow the same ideas as the proof of \cite[Proposition 5.1]{GdI3}. First, for $H$, we have that the sequence of alternating numbers $p_0,q_1,p_1,q_2, \ldots$ is a \emph{chain sequence}. Let us call $\alpha_n, n\geq1,$ the sequence of partial numerators. Then $\alpha_n=(1-m_{n-1})m_n$, where
$$
m_{2n}=\frac{n+t}{2n+2t+\alpha+\beta+1},\quad m_{2n+1}=\frac{n+t+\beta+1}{2n+2t+\alpha+\beta+2}.
$$
According to \cite[Theorem 3.1]{Ch} we have that
$$
H=m_0+\frac{1-m_0}{1+L},\quad L=\sum_{n=1}^\infty\prod_{k=1}^n\frac{m_k}{1-m_k}.
$$
It is possible to show that $L$ is convergent as long as $\alpha>0$, in which case we have $L=(\beta+t+1)/\alpha$. A direct computation gives the value of $H$ in \eqref{HHp}. On the other hand, for $H'$, we have again that the sequence $\alpha_n'=(1-m'_{n-1})m'_n, n\geq1$ of alternating numbers $q_0,p_{-1},q_{-1},p_{-2}, \ldots$ is a chain sequence where
$$
m'_{2n}=\frac{-n+t+\alpha+\beta+1}{-2n+2t+\alpha+\beta+1},\quad m'_{2n+1}=\frac{-n+t+\alpha}{-2n+2t+\alpha+\beta}.
$$
Therefore 
$$
1-H'=m'_0+\frac{1-m'_0}{1+L'},\quad L'=\sum_{n=1}^\infty\prod_{k=1}^n\frac{m'_k}{1-m'_k}.
$$
It is possible to show that $L'$ is convergent as long as $\alpha>0$, in which case we have $L'=-(\alpha+t)/\alpha$. A direct computation gives the value of $H'$ in \eqref{HHp}. Finally, a direct computation using \eqref{xysr1} and \eqref{ULd} gives \eqref{coefTTRR}.
\end{proof}

\subsection{Stochastic LU factorization}

Consider now the LU factorization of the stochastic matrix $P$ given by \eqref{QZ} in the following way
\begin{equation}\label{QZLU}
P=\left(\begin{array}{ccc|ccc}
\ddots&\ddots&&&\\
&\tilde t_{-1}&\tilde s_{-1}&0&\\
\hline
&&\tilde t_0&\tilde s_0&0\\
&&&\tilde t_1&\tilde s_1&0\\
&&&&\ddots&\ddots
\end{array}
\right)\left(\begin{array}{cc|cccc}
\ddots&\ddots&&\\
0&\tilde y_{-1}&\tilde x_{-1}&&\\
\hline
&0&\tilde y_0&\tilde x_0&\\
&&0&\tilde y_1&\tilde x_1\\
&&&&\ddots&\ddots
\end{array}
\right)=\widetilde P_L\widetilde P_U,
\end{equation}
where again $\widetilde P_L$ and $\widetilde P_U$ are also stochastic matrices, i.e. all entries of $\tilde P_L$ and $\tilde P_U$ are nonnegative and
\begin{equation*}\label{stplu}
\tilde t_n+ \tilde s_n=1, \quad \tilde y_n +\tilde x_n=1,\quad  n \in \ZZ.
\end{equation*}
Now, a direct computation shows that
\begin{align}
\nonumber p_n&=\tilde s_n \tilde x_n,\\
\label{LUd}r_n&=\tilde t_n\tilde x_{n-1}+\tilde s_n\tilde y_n,\quad n\in\ZZ,\\
\nonumber q_n&=\tilde t_n\tilde y_{n-1}.
\end{align}
Again, typically this factorization depends on a free parameter $\tilde r_0$ and applying \cite[Theorem 2.2]{dIJ1} we need that $H'\leq\tilde r_0\leq H$ in order to have a stochastic LU factorization, where $H$ and $H'$ are defined by \eqref{cfUL}.

\begin{theorem}\label{thm2}
Assume that $\alpha>0$. Then $H=H'$ is convergent to \eqref{HHp} and there exists only one value of the parameter $\tilde r_0$ ($\tilde r_0=H$) such that we obtain a stochastic LU factorization of the form \eqref{QZLU} and the coefficients of each of the factors $\widetilde{P}_L$ and $\widetilde{P}_U$ are given by
\begin{equation}\label{xysr2}
\tilde y_n=t_{n+1},\quad \tilde x_n=s_{n+1},\quad \tilde s_n=x_n,\quad \tilde t_n=y_n\quad n\in\ZZ,
\end{equation}
where $x_n,y_n,s_n,t_n$ are defined by \eqref{xysr1}.
\end{theorem}

\begin{proof}
Identical to the proof of Theorem \ref{thm1} but using \eqref{LUd}.
\end{proof}

\begin{remark}
It is possible to see that if we choose $t$ in the ranges described in Proposition \ref{prop1} for the regions $\bm A_1, \bm A_2, \bm B_1, \bm B_2$, then we will always have that $0<H<1$ and a stochastic UL or LU factorization will always be possible. On the contrary, in regions $\bm C_1, \bm C_2, \bm D_1, \bm D_2$ it is not possible to give a stochastic UL or LU factorization since $\alpha<0$ and the convergence of $H$ and/or $H'$ is not guaranteed.
\end{remark}

\begin{remark}
In \cite{dIJ2} we introduced the so-called \emph{reflecting-absorbing} (or RA) factorization. Instead of dividing the bilateral birth-death chain into a pure-birth and a pure-death process (as in the case of UL and LU factorization) we considered the first of the factors to be an absorbing process to the state 0 and the second factor to be a reflecting process from the state 0. In this case there will be 2 free parameters ($a$ and $x_0$) and in order to guarantee a stochastic RA factorization we need $a\geq H'$ and $x_0\geq1-H$. Since $a+x_0=1$ ($H=H'$) that means that $y_0=0,$ so the RA factorization will not be possible (see \cite{dIJ2} for details).
\end{remark}

\section{Stochastic Darboux transformations and the associated spectral matrices}\label{sec5}

Once we have a UL (or LU) factorization we can perform what is called a \emph{discrete Darboux transformation}. The Darboux transformation has a long history but probably the first reference of a discrete Darboux transformation like we study here appeared in \cite{MS} in connection with the Toda lattice. If $P=P_UP_L$ as in \eqref{QZUL}, then by inverting the order of the factors we obtain another tridiagonal matrix of the form $\widetilde P=P_LP_U$. The new coefficients of the matrix $\widetilde P$ are given by
\begin{align*}
\nonumber\tilde{p}_n&=s_nx_{n},\\
\label{coeffDLU}\tilde{r}_n&=t_nx_{n-1}+s_ny_n,\quad n\in\ZZ,\\
\nonumber\tilde{q}_n&=t_ny_{n-1}.
\end{align*}
The matrix $\widetilde{P}$ is stochastic, since the multiplication of two stochastic matrices is again a stochastic matrix. Actually, using \eqref{xysr1}, we have that $\tilde{p}_n, \tilde{r}_n, \tilde{q}_n$ are the coefficients $p_n, r_n, q_n$ in \eqref{coefTTRR} replacing $\alpha$ by $\alpha-1$. In other words,
$$
\tilde p_n=\left.p_n\right|_{\alpha=\alpha-1},\quad \tilde r_n=\left.r_n\right|_{\alpha=\alpha-1},\quad\tilde q_n=\left.q_n\right|_{\alpha=\alpha-1}.
$$
Therefore the new discrete-time bilateral birth-death chain $\{\widetilde Z_t : t=0,1,\ldots\}$ on the integers $\ZZ$ with coefficients $(\tilde{p}_n)_{n\in\ZZ}$, $(\tilde{r}_n)_{n\in\ZZ}$ and $(\tilde{q}_n)_{n\in\ZZ}$ is the same as the original birth-death chain $Z_t$ but replacing the parameter $\alpha$ by $\alpha-1$. The spectral matrix $\widetilde{\bm\Psi}$ associated with $\widetilde P$ is then given by 
$$
\widetilde{\bm\Psi}(x)=\left.\bm\Psi(x)\right|_{\alpha=\alpha-1},
$$ 
where $\bm\Psi$ is the spectral matrix \eqref{psiw}. Observe that $\widetilde{\bm\Psi}$ is well-defined on $[0,1]$ since we are assuming that $\alpha>0$ in order to have a stochastic UL factorization (see Theorem \ref{thm1}). Also $\widetilde{\bm\Psi}$ does not have a discrete part, just as $\bm\Psi$. Another representation of $\widetilde{\bm\Psi}$ comes from \cite[Theorem 3.5]{dIJ1}, given in terms of a \emph{Geronimus transformation} of the spectral matrix $\bm\Psi$. This representation is given by
$$
\widetilde{\bm\Psi} (x)=\bm S_0(x) \bm\Psi_S (x) \bm S_0^*(x),
$$
where (see \eqref{xysr1})
$$
\bm S_0(x)=\begin{pmatrix} s_0& t_0 \\ -\D\frac{x_{-1}s_0}{y_{-1}}& \D\frac{x-x_{-1}t_0}{y_{-1}}\end{pmatrix}=\begin{pmatrix} \D\frac{t+\alpha+\beta}{2t+\alpha+\beta}& \D\frac{t}{2t+\alpha+\beta} \\ -\D\frac{(t+\beta)(t+\alpha+\beta)}{(2t+\alpha+\beta)(t+\alpha-1)}& \D\frac{2t+\alpha+\beta-1}{t+\alpha-1}x-\frac{t(t+\beta)}{(2t+\alpha+\beta)(t+\alpha-1)}\end{pmatrix},
$$
and
$$
\bm\Psi_S(x)=\frac{y_0}{s_0}\frac{\bm\Psi(x)}{x}+ \left[\begin{pmatrix}1/s_0&0 \\ 0 &1/t_0 \end{pmatrix}- \frac{y_0}{s_0}\bm M_{-1}\right] \delta_0(x).
$$
Here $\bm M_{-1}=\int_{0}^{1} x^{-1}\bm\Psi(x)dx$. Since $\widetilde{\bm\Psi}$ does not have discrete spectrum that means that the matrix in front of $\delta_0(x)$ in $\widetilde{\bm\Psi}$ should be the null matrix. In \cite[Theorem 3.5]{dIJ1} it is assumed that $\bm M_{-1}$ is well-defined (entry by entry). But a simple computation shows that in general $\bm M_{-1}$ is not well-defined for $\alpha>0$. However, this assumption is actually too restrictive. It is enough to assume that $\bm S_0(0) \bm M_{-1} \bm S_0^*(0)$ is well-defined. $\bm S_0(0)$ is in fact a singular matrix and after simplifications it turns out that the integral of $x^{-1}\bm S_0(0) \bm\Psi(x) \bm S_0^*(0)$ over $[0,1]$ is always well-defined. In fact we have
$$
\frac{y_0}{s_0}\bm S_0(0) \bm M_{-1} \bm S_0^*(0)=\begin{pmatrix}
1&-\D\frac{t+\beta}{t+\alpha-1}\\-\D\frac{t+\beta}{t+\alpha-1}&\left(\D\frac{t+\beta}{t+\alpha-1}\right)^2
\end{pmatrix}.
$$
A simple computation shows that 
$$
\bm S_0(0)\begin{pmatrix}1/s_0&0 \\ 0 &1/t_0 \end{pmatrix}\bm S_0^*(0)
$$ 
is the same matrix as before so we have that the matrix in front of $\delta_0(x)$ in $\widetilde{\bm\Psi}$ is the null matrix. Concluding, another representation of the spectral matrix $\widetilde{\bm\Psi}$ is given by
$$
\widetilde{\bm\Psi}(x)=\frac{y_0}{s_0}\bm S_0(x) \frac{\bm\Psi(x)}{x}\bm S_0^*(x).
$$
Finally, if we construct the matrix-valued polynomials $(\widetilde{\bm Q}_n)_{n\geq0},$ in the same way as we did in \eqref{2QMM} but using $\widetilde P$ instead of $P$, we have that 
\begin{equation*}\label{qusnorms}
\int_{0}^{1} \widetilde{\bm Q}_n (x) \widetilde{\bm\Psi} (x) \widetilde{\bm Q}_m^*(x) dx=\begin{pmatrix} 1/\tilde \pi_n &0\\ 0& 1/\tilde \pi_{-n-1}\end{pmatrix} \delta_{nm},
\end{equation*}
where $(\tilde \pi_n)_{n\in \mathbb{Z}}$ are the potential coefficients $(\pi_n)_{n\in \mathbb{Z}}$ defined by \eqref{potcoeff} replacing $\alpha$ by $\alpha-1$. The bilateral birth-death chain $\{\widetilde Z_t : t=0,1,\ldots\}$ associated with $\widetilde P$ will have a similar Karlin-McGregor representation formula as in \eqref{KMcG1} and it is always null recurrent  as the original one (see the paragraph before Remark \ref{rem32}).

\bigskip

The same can be done for the LU factorization \eqref{QZLU} of the form $P=\widetilde P_L\widetilde P_U$. By inverting the order of the factors we obtain another tridiagonal matrix of the form $\widehat P=\widetilde P_U\widetilde P_L$. The new coefficients of the matrix $\widehat P$ are given by (see \eqref{xysr2})
\begin{equation*}\label{coeffDUL}
\begin{split}
\hat{p}_n&=s_{n+1}x_{n+1},\\
\hat{r}_n&=s_{n+1}y_{n+1}+t_{n+1}x_n,\quad n\in\ZZ.\\
\hat{q}_n&=t_{n+1}y_n.
\end{split}
\end{equation*}
Again, the matrix $\widehat{P}$ is stochastic, so we have a new bilateral birth-death chain $\{\widehat Z_t : t=0,1,\ldots\}$ on the integers $\ZZ$ with coefficients $(\hat{p}_n)_{n\in\ZZ}$, $(\hat{r}_n)_{n\in\ZZ}$ and $(\hat{q}_n)_{n\in\ZZ}$. A simple computation, using \eqref{xysr2} and \eqref{xysr1}, gives
$$
\hat p_n=\left.p_{n+1}\right|_{\alpha=\alpha-1},\quad \hat r_n=\left.r_{n+1}\right|_{\alpha=\alpha-1},\quad\hat q_n=\left.q_{n+1}\right|_{\alpha=\alpha-1}.
$$
Now we do not have a priori a candidate for the spectral matrix $\widehat{\bm\Psi}$ associated with $\widehat P$. The previous shifted coefficients give a spectral matrix which is not as easily identificable as the previous case $\widetilde{\bm\Psi}$ of the UL factorization. However we can still apply \cite[Theorem 3.9]{dIJ1} to compute $\widehat{\bm\Psi}$ in terms of a Geronimus transformation of the spectral matrix $\bm\Psi$. $\widehat{\bm\Psi}$ is given by
$$
\widehat{\bm\Psi} (x)=\bm T_0(x) \bm\Psi_T (x) \bm T_0^*(x),
$$
where (see \eqref{xysr2} and \eqref{xysr1})
$$
\bm T_0(x)=\begin{pmatrix}  \D\frac{x-s_0y_0}{x_{0}}&  -\D\frac{y_0t_0}{x_0} \\ s_0&t_0\end{pmatrix}=\begin{pmatrix} \D\frac{2t+\alpha+\beta+1}{t+\beta+1}x-\frac{(t+\alpha)(t+\alpha+\beta)}{(2t+\alpha+\beta)(t+\beta+1)} &-\D\frac{t(t+\alpha)}{(2t+\alpha+\beta)(t+\beta+1)} \\ \D\frac{t+\alpha+\beta}{2t+\alpha+\beta} &\D\frac{t}{2t+\alpha+\beta} \end{pmatrix},
$$
and
$$
\bm\Psi_T(x)=\frac{x_0}{t_1}\frac{\bm\Psi(x)}{x}+ \left[\frac{x_0}{y_0t_1}\begin{pmatrix}1&0 \\ 0 &s_0/t_0 \end{pmatrix}- \frac{x_0}{t_1}\bm M_{-1}\right] \delta_0(x).
$$
The spectrum of $\widehat{P}$ is the same as the spectrum of $P$ (we are only shifting the coefficients one step forward and replacing $\alpha$ by $\alpha-1$). Therefore we should expect, as in the case of UL factorization, that the matrix in front of $\delta_0(x)$ in $\widehat{\bm\Psi}$ is the null matrix. Proceeding as before we have
$$
\frac{x_0}{t_1}\bm T_0(0) \bm M_{-1} \bm T_0^*(0)=\frac{(2t+\alpha+\beta+2)(t+\alpha+\beta)(t+\alpha)}{(t+1)(2t+\alpha+\beta)(t+\beta+1)}\begin{pmatrix}
1&-\D\frac{t+\beta+1}{t+\alpha}\\-\D\frac{t+\beta+1}{t+\alpha}&\left(\D\frac{t+\beta+1}{t+\alpha}\right)^2
\end{pmatrix},
$$
which is the same matrix as 
$$
\D\frac{x_0}{y_0t_1}\bm T_0(0)\begin{pmatrix}1&0 \\ 0 &s_0/t_0 \end{pmatrix}\bm T_0^*(0).
$$
Therefore the matrix in front of $\delta_0(x)$ in $\widehat{\bm\Psi}$ is the null matrix. As a consequence we have that the spectral matrix $\widehat{\bm\Psi}$ is given by
$$
\widehat{\bm\Psi}(x)=\frac{x_0}{t_1}\bm T_0(x) \frac{\bm\Psi(x)}{x}\bm T_0^*(x).
$$
Finally, if we construct the matrix-valued polynomials $(\widehat{\bm Q}_n)_{n\geq0}$ in the same way as we did in \eqref{2QMM} but using $\widehat P$ instead of $P$, we have that 
\begin{equation*}\label{qusnorms}
\int_{0}^{1} \widehat{\bm Q}_n (x) \widehat{\bm\Psi} (x) \widehat{\bm Q}_m^*(x) dx=\begin{pmatrix} 1/\hat \pi_n &0\\ 0& 1/\hat \pi_{-n-1}\end{pmatrix} \delta_{nm},
\end{equation*}
where now
$$
\hat\pi_n=\frac{(\alpha+\beta+t+1)_n(-t-1)_{-n}(2n+2t+\alpha+\beta+2)}{(\alpha+t+1)_n(-\beta-t-1)_{-n}(2t+\alpha+\beta+2)},\quad n\in\mathbb{Z}.
$$
The birth-death chain $\{\widehat Z_t : t=0,1,\ldots\}$ associated with $\widehat P$ will have a similar Karlin-McGregor representation formula as in \eqref{KMcG1} and again it is always null recurrent.

\begin{remark}
The new families of matrix-valued orthogonal polynomials $(\widetilde{\bm Q}_n)_{n\geq0}$  and $(\widehat{\bm Q}_n)_{n\geq0}$ constructed from the UL and LU Darboux transformations are also eigenfunctions of a matrix-valued second-order differential operator of the form \eqref{sode}. The coefficients and eigenvalues of the differential operator for the first family $(\widetilde{\bm Q}_n)_{n\geq0}$ are given by \eqref{CUVL} replacing $\alpha$ by $\alpha-1$. On the other hand, the coefficients and eigenvalues of the differential operator for the second family $(\widehat{\bm Q}_n)_{n\geq0}$ are given by \eqref{CUVL} replacing $\alpha$ by $\alpha-1$ and $t$ by $t+1$. In particular both families are \emph{bispectral}. In the scalar case, and choosing special values of the parameters involved, the order of the differential operator after a Darboux transformation is always higher than 2. In the matrix case we have that after one step of the Darboux transformation, the order of the differential operator can be the same as the original one. This phenomenon is not new and appeared for the first time in \cite{DdI2} using a method different than the Darboux transformation. For other examples of the bispectral property following a Darboux transformation see \cite{G3,GdI4}.
\end{remark}

\section{An urn model for the associated Jacobi polynomials}\label{sec6}

We now give an urn model for the associated Jacobi polynomials studied in the previous sections. For simplicity, we will restrict the parameters $\alpha$ and $\beta$ to the region $\bm A_1=\{\beta-\alpha+1>0,\alpha>-\beta,\beta<0\}$ given in Proposition \ref{prop1} (for the rest of regions we can proceed in a similar way). We recall from that proposition that, in order to have a stochastic matrix $P$, we need to choose the parameter $t$ in the real set \eqref{setA1}. Since $-1<\alpha<1,-1<\beta<0$ and $t$ is a real parameter, in order to find an urn model including numbers of nonnegative blue or red balls, in this section we will assume that
\begin{equation}\label{chov}
\alpha=\frac{1}{A},\quad\beta=-\frac{1}{B},\quad t=\frac{1}{T}+K,\quad A,B,T\in\mathbb{Z}_{\geq2},\quad K\in\mathbb{Z}_{\geq0}.
\end{equation}
On one side, the restriction on the region $\bm A_1$ is equivalent to $AB>A+B$ and $A<B$. On the other side, the restriction \eqref{setA1} gives two possibilities:
\begin{itemize}
\item $n<t<n-\beta,n\in\mathbb{Z}$. Substituting \eqref{chov} in the previous inequalities we get that we need $n<1/T+K<n+1/B,n\in\mathbb{Z}$. It turns out that, since $A,B,T\geq2$ and $K\geq0$ are nonnegative integers,  the only choice of $n\in\mathbb{Z}$ such that the previous both inequalities hold is for $n=K$, in which case we need that $T>B$. If $n<K$ then the first inequality is not possible and if $n>K$ then the second inequality does not hold.
\item $n-\alpha<t<n-\alpha-\beta,n\in\mathbb{Z}$. Substituting \eqref{chov} in the previous inequalities we get that we need $n-1/A<1/T+K<n-1/A+1/B,n\in\mathbb{Z}$. Again, it turns out that, since $A,B,T\geq2$ and $K\geq0$ are nonnegative integers,  the only choice of $n\in\mathbb{Z}$ such that the previous both inequalities hold is for $n=K$, in which case we need that $B<AT/(A+T)$. If $n<K$ then the first inequality is not possible and if $n>K$ then the second inequality does not hold.
\end{itemize}
We will choose $T$ according to the first possibility (for the second we can proceed in a similar way). In summary, our nonnegative parameters $A,B,T,K$ will be restricted to
\begin{equation}\label{condspar}
A,B,T\geq2,\quad K\geq0, \quad AB>A+B,\quad\mbox{and}\quad A<B<T.
\end{equation}

We focus now on the case of the UL stochastic factorization $P=P_UP_L$ in \eqref{QZUL} with coefficients $x_n,y_n,s_n,t_n,n\in\mathbb{Z}$ given by \eqref{xysr1}. Substituting \eqref{chov} in these coefficients we obtain
\begin{equation}\label{XYST1}
\begin{split}
y_n&=\frac{B(AT(n+K)+A+T)}{ABT(2n+2K+1)+2AB+T(B-A)},\quad x_n=\frac{A(BT(n+K+1)+B-T}{ABT(2n+2K+1)+2AB+T(B-A)},\\
s_n&=\frac{ABT(n+K)+AB+T(B-A)}{2ABT(n+K)+2AB+T(B-A)},\qquad\;\;\;\;\, t_n=\frac{AB(T(n+K)+1)}{2ABT(n+K)+2AB+T(B-A)},
\end{split}\qquad n\in\ZZ.
\end{equation}
To simplify the notation, let us call 
\begin{equation}\label{XYSTc}
\begin{split}
Y_n&=B(AT(n+K)+A+T),\quad X_n=A(BT(n+K+1)+B-T,\\
S_n&=ABT(n+K)+AB+T(B-A),\quad T_n=AB(T(n+K)+1),
\end{split}\quad n\in\mathbb{Z},
\end{equation}
so that we have 
$$
y_n=\frac{Y_n}{X_n+Y_n},\quad x_n=\frac{X_n}{X_n+Y_n},\quad s_n=\frac{S_n}{S_n+T_n},\quad t_n=\frac{T_n}{S_n+T_n},\quad n\in\mathbb{Z}.
$$
\begin{lemma}\label{lemxyrs}
Assume that we have that $A,B,T,K$ are nonnegative integers satisfying \eqref{condspar}. If $n+K\geq0$ then $Y_n,X_n,S_n,R_n\geq0$ for all $n\in\mathbb{Z}$ and if $n+K<0$ then $Y_n,X_n,S_n,R_n<0$ for all $n\in\mathbb{Z}$.
\end{lemma}
\begin{proof}
From \eqref{XYSTc} we have that $Y_n,X_n,S_n,R_n\geq0$ if
$$
n+K\geq\begin{cases}
\D-\frac{1}{T}-\frac{1}{A},&\mbox{for $Y_n$},\\\\
\D-1-\frac{1}{T}+\frac{1}{B},&\mbox{for $X_n$},\\\\
\D-\frac{1}{T}-\frac{1}{A}+\frac{1}{B},&\mbox{for $S_n$},\\\\
\D-\frac{1}{T}&\mbox{for $T_n$}.
\end{cases}
$$
A straightforward computation using \eqref{condspar} shows that
$$
-1<-1-\frac{1}{T}+\frac{1}{B}<-\frac{1}{T}-\frac{1}{A}+\frac{1}{B}<-\frac{1}{T}-\frac{1}{A}<-\frac{1}{T}<0.
$$
From the previous inequalities it is now easy to see that if $n+K\geq0$ then $Y_n,X_n,S_n,R_n\geq0$, while if $n+K<0$ then $Y_n,X_n,S_n,R_n<0$ for all $n\in\mathbb{Z}$.
\end{proof}

Let $\{Z_t: t= 0, 1,\ldots\}$ be the bilateral birth-death chain associated with the transition probability matrix $P$ \eqref{QZ}. 
Consider the UL stochastic factorization $P=P_UP_L$ in \eqref{QZUL}. We have that each one of the matrices $P_U$ and $P_L$ will give rise to an experiment in terms of an urn model, which we call Experiment 1 and Experiment 2, respectively. Let us call $\{Z_t^{(i)}: t= 0, 1,\ldots\}, i=1,2,$ the chains associated with $P_U$ and $P_L$, respectively. At times $t = 0, 1,\ldots,$ the state $n\in\mathbb{Z}$ in each of these chains will be given by the \emph{number of blue balls minus the number of red balls}. Therefore we may have nonnegative and negative integer states. We finally assume that the urn sits in a bath consisting of an infinite number of blue and red balls.

At the beginning of every experiment for $t = 0, 1,\ldots,$ we have to decide how many blue and red balls are going to be in the urn. This will depend on the state $n\in\mathbb{Z}$ according to the following rule:
\begin{enumerate}
\item If the initial state $n$ satisfies $n\geq-K$, then we initially put in the urn $n+K$ blue balls and $K$ red balls.
\item If the initial state $n$ satisfies $n<-K$, then we initially put in the urn $K$ blue balls and $-n+K$ red balls.
\end{enumerate}

\medskip

Experiment 1 (for $P_U$) will give a pure-birth chain on $\mathbb{Z}$ with diagram given by
\begin{center}
\includegraphics[scale=.3]{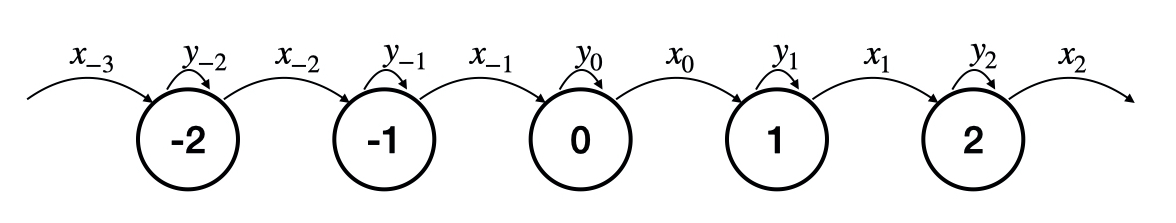}
\end{center}
Initially, $Z_0^{(1)}=n$. On one hand, if $n\geq-K$, then we place $n+K$ blue balls and $K$ red balls in the urn. After that we remove/add balls until we have $X_n$ blue balls and $Y_n$ red balls (both nonnegative integers by Lemma \ref{lemxyrs}). Draw one ball from the urn at random with the uniform distribution. We have two possibilities:
\begin{itemize}
\item If we get a blue ball (with probability $x_n$ in \eqref{XYST1}) then we remove/add balls until we have $n+K+1$ blue balls and $K$ red balls in the urn and start over. Then we have $Z_1^{(1)}=n+1$.
\item If we get a red ball (with probability $y_n$ in \eqref{XYST1}) then we remove/add balls until we have $n+K$ blue balls and $K$ red balls in the urn and start over. Then we have $Z_1^{(1)}=n$.
\end{itemize}
On the other hand, if $n<-K$, then we place $K$ blue balls and $-n+K$ red balls in the urn. After that we remove/add balls until we have $-X_n$ blue balls and $-Y_n$ red balls (both nonnegative integers by Lemma \ref{lemxyrs}). Draw one ball from the urn at random with the uniform distribution. We have two possibilities:
\begin{itemize}
\item If we get a blue ball (with probability $x_n$ in \eqref{XYST1}) then we remove/add balls until we have $K+1$ blue balls and $-n+K$ red balls in the urn and start over. Then we have $Z_1^{(1)}=n+1$.
\item If we get a red ball (with probability $y_n$ in \eqref{XYST1}) then we remove/add balls until we have $K$ blue balls and $-n+K$ red balls in the urn and start over. Then we have $Z_1^{(1)}=n$.

\end{itemize}

\medskip

Experiment 2 (for $P_L$) will give a pure-death chain on $\mathbb{Z}$ with diagram given by
\begin{center}
\includegraphics[scale=.3]{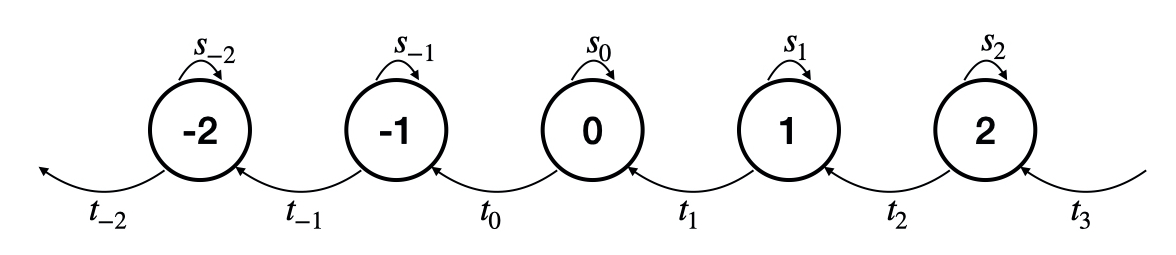}
\end{center}
Initially, $Z_0^{(2)}=n$. On one hand, if $n\geq-K$, then we place $n+K$ blue balls and $K$ red balls in the urn. After that we remove/add balls until we have $S_n$ blue balls and $T_n$ red balls (both nonnegative integers by Lemma \ref{lemxyrs}). Draw one ball from the urn at random with the uniform distribution. We have two possibilities:
\begin{itemize}
\item If we get a blue ball (with probability $s_n$ in \eqref{XYST1}) then we remove/add balls until we have $n+K$ blue balls and $K$ red balls in the urn and start over. Then we have $Z_1^{(2)}=n$.
\item If we get a red ball (with probability $t_n$ in \eqref{XYST1}) then we remove/add balls until we have $n+K$ blue balls and $K+1$ red balls in the urn and start over. Then we have $Z_1^{(2)}=n-1$.
\end{itemize}
On the other hand, if $n<-K$, then we place $K$ blue balls and $-n+K$ red balls in the urn. After that we remove/add balls until we have $-S_n$ blue balls and $-T_n$ red balls (both nonnegative integers by Lemma \ref{lemxyrs}). Draw one ball from the urn at random with the uniform distribution. We have two possibilities:
\begin{itemize}
\item If we get a blue ball (with probability $s_n$ in \eqref{XYST1}) then we remove/add balls until we have $K$ blue balls and $-n+K$ red balls in the urn and start over. Then we have $Z_1^{(2)}=n$.
\item If we get a red ball (with probability $t_n$ in \eqref{XYST1}) then we remove/add balls until we have $K$ blue balls and $-n+K+1$ red balls in the urn and start over. Then we have $Z_1^{(2)}=n-1$.
\end{itemize}

The urn model for $P$ (on $\mathbb{Z}$) is obtained by repeatedly alternating Experiments 1 and 2 in that order. For the benefit of the reader, we have included Figures \ref{fig2}, \ref{fig3} and \ref{fig4} below, that explains the 3 possible situations that we can find, given the value of the initial state $n\in\mathbb{Z}$. Figures \ref{fig2} and \ref{fig3} are the diagrams for the cases where $n\geq-K$ and $n<-K-1$, respectively, where the number of blue and red balls at the beginning of Experiment 2 does not change. Figure \ref{fig4}, for $n=-K-1$ is the only case where at the beginning of Experiment 2 (when the state is $-K$) we need to place a different number of blue and red balls.

\begin{figure}[h]
\begin{center}
\vspace{-0.0cm}
\includegraphics[height=10cm]{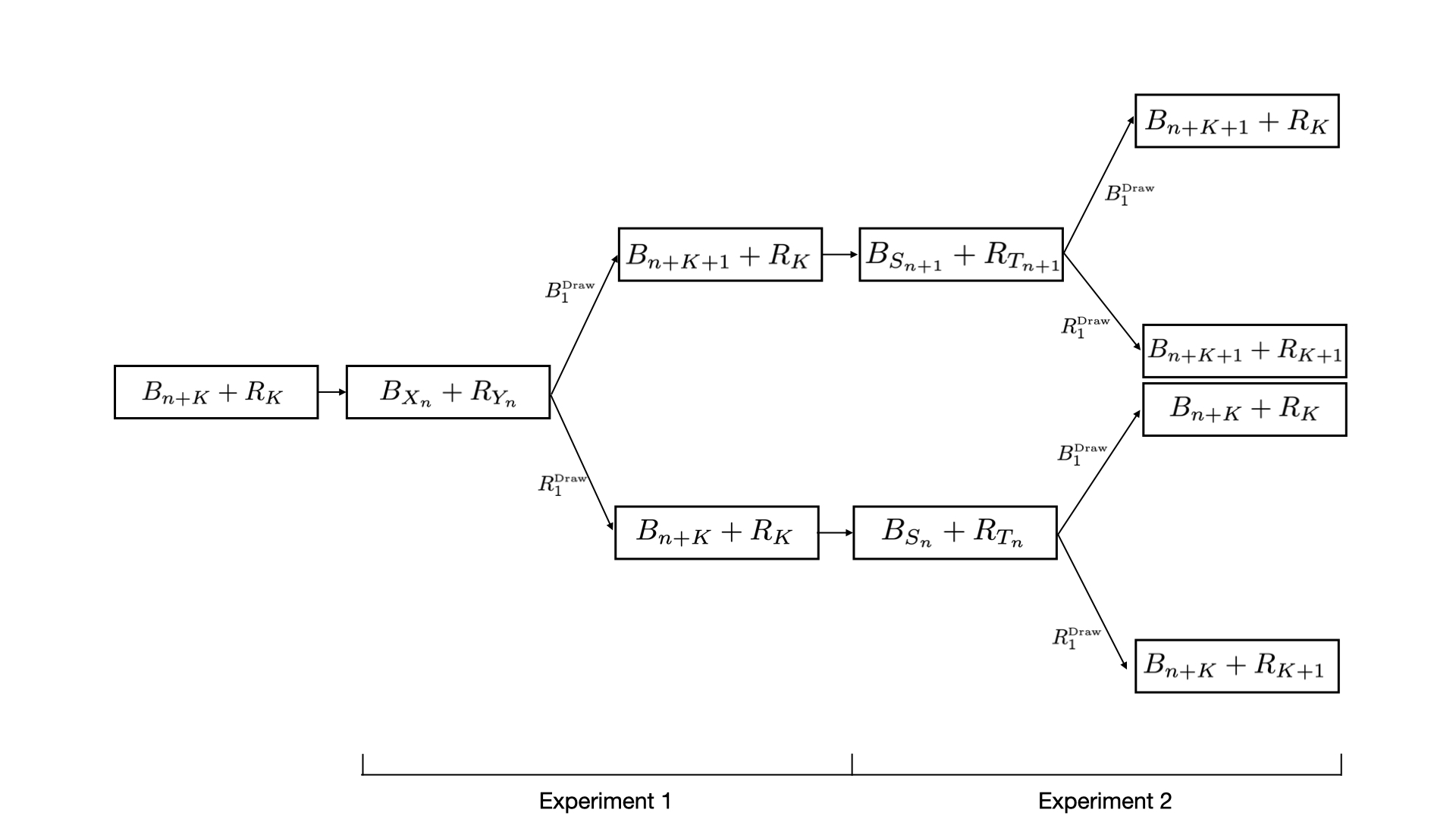}
\end{center}
\vspace{-0.2cm}
	\caption{A schematic of Experiments 1 and 2 when $n\geq-K$. The boxed regions represent the number of blue balls $B_{b}$ and red balls $R_{r}$ contained within the urn, so that the state of the system is $n=b-r$. When a ball is drawn from an urn, this is indicated by $B_1^{\tiny{\mbox{Draw}}}$ or $R_1^{\tiny{\mbox{Draw}}}$.}
\label{fig2}
\end{figure}

\begin{figure}[h]
\begin{center}
\vspace{-0.0cm}
\includegraphics[height=10cm]{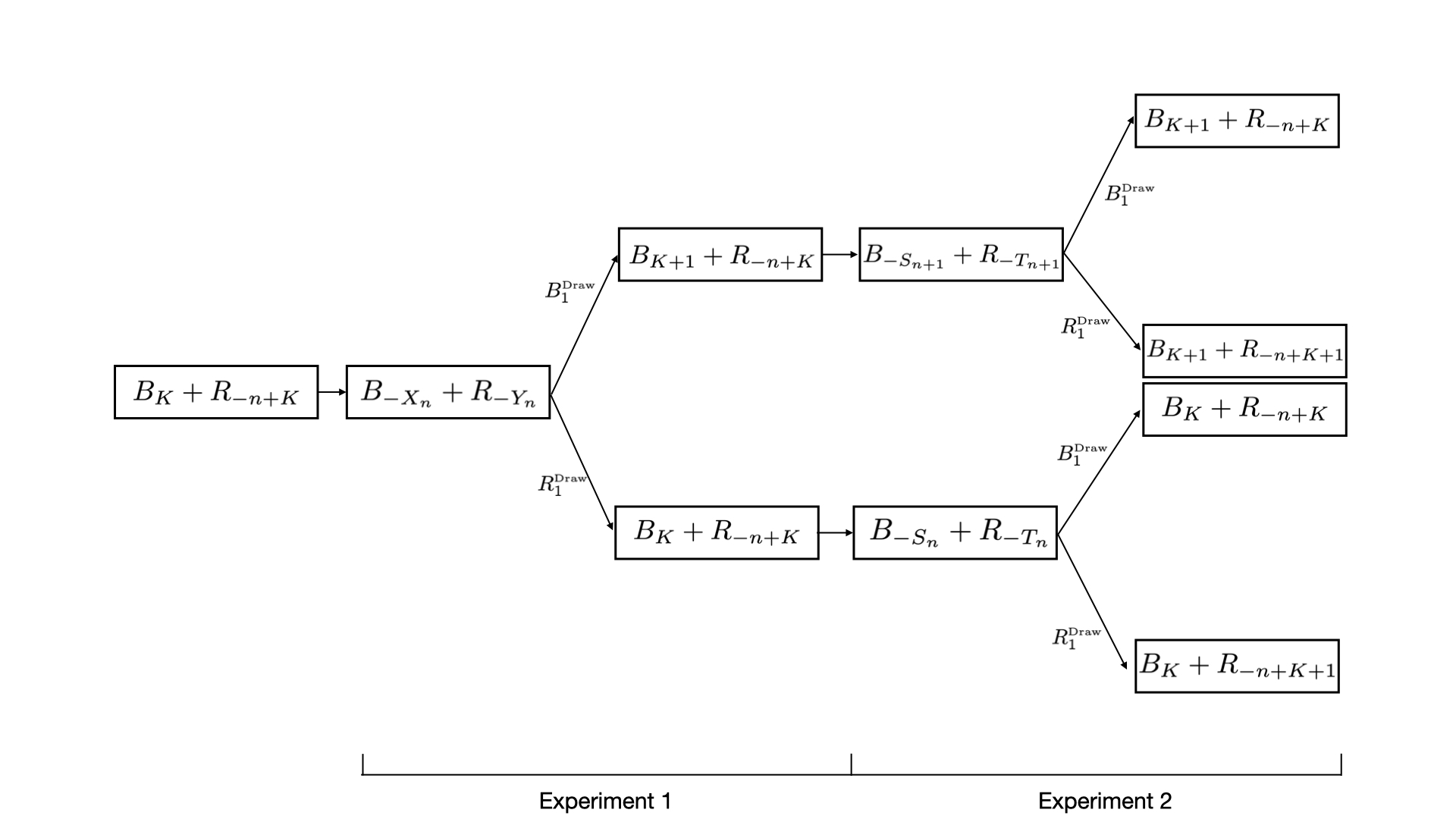}
\end{center}
\vspace{-0.2cm}
	\caption{A schematic of Experiments 1 and 2 when $n<-K-1$. The boxed regions represent the number of blue balls $B_{b}$ and red balls $R_{r}$ contained within the urn, so that the state of the system is $n=b-r$. When a ball is drawn from an urn, this is indicated by $B_1^{\tiny{\mbox{Draw}}}$ or $R_1^{\tiny{\mbox{Draw}}}$.}
\label{fig3}
\end{figure}

\begin{figure}[h]
\begin{center}
\vspace{-0.0cm}
\includegraphics[height=10cm]{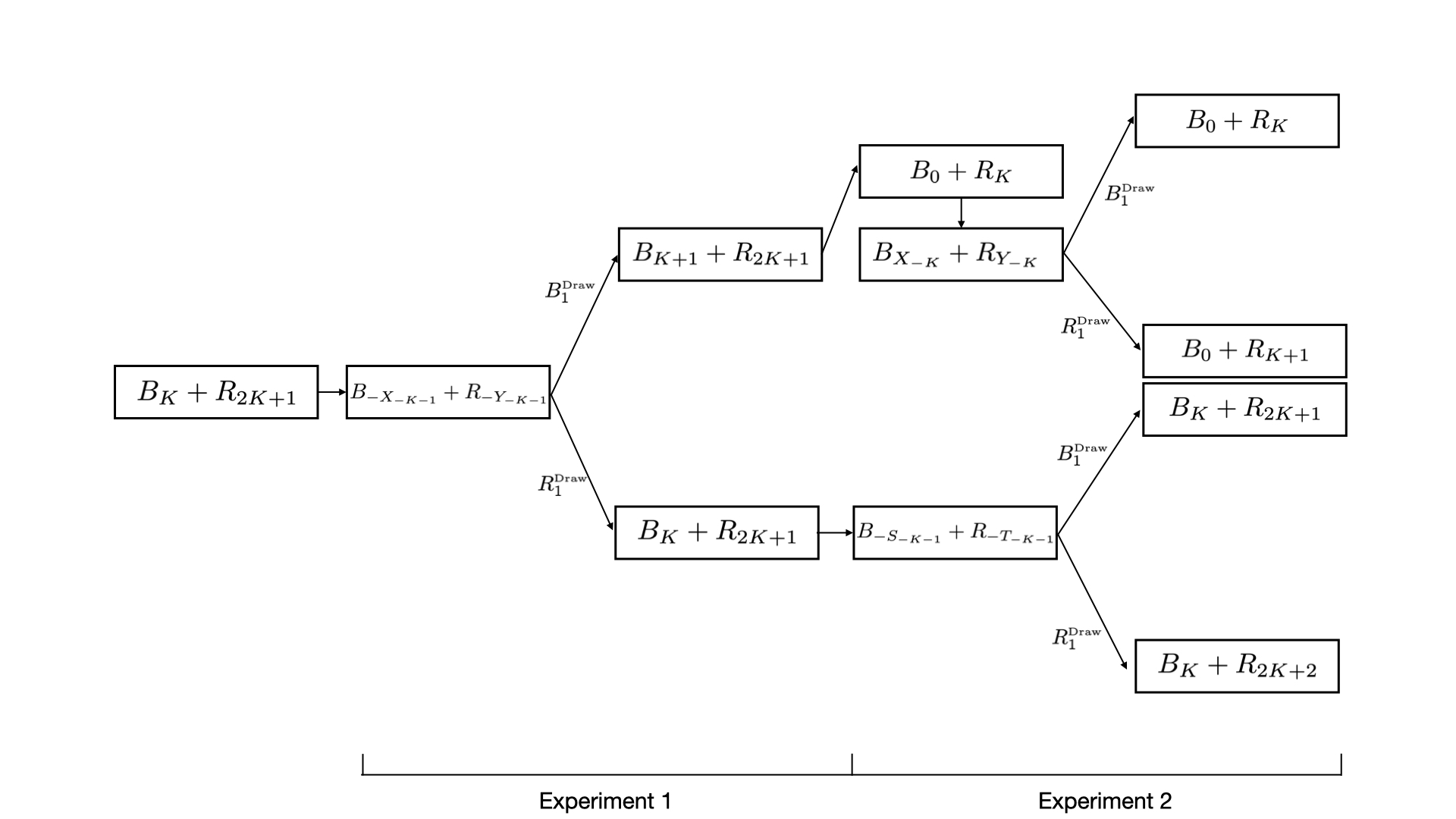}
\end{center}
\vspace{-0.2cm}
	\caption{A schematic of Experiments 1 and 2 when $n=-K-1$. The boxed regions represent the number of blue balls $B_{b}$ and red balls $R_{r}$ contained within the urn, so that the state of the system is $n=b-r$. When a ball is drawn from an urn, this is indicated by $B_1^{\tiny{\mbox{Draw}}}$ or $R_1^{\tiny{\mbox{Draw}}}$.}
\label{fig4}
\end{figure}

In a similar way the urn model for the Darboux transformation $\widetilde{P}=P_LP_U$ proceeds in the reversed order. Similar urn models can be derived for the LU factorization \eqref{QZLU} with small modifications.


\begin{thebibliography}{99}


\bibitem{AWi} Askey, R. and Wimp, J., {\em Associated Laguerre polynomials},  Proc. Roy. Soc. Edinburgh \textbf{96} (1984), 15--37.

\bibitem{AIB} Assche, W., \emph{Encyclopedia of Special Functions: The Askey-Bateman Project}, vol. I (M. Ismail, Ed.), Cambridge University Press, 2020.

\bibitem{BarDic} Barrucand, P. and Dickinson, D.,\emph{On the associated Legendre polynomials}, in: Orthogonal Expansions and their Continual Analogues (D.T. Haimo, ed.), Southern Illinois U. Press, Edwardsville, 1968, 43--50.

\bibitem{Ber} Berezans'kii, Ju M., \emph{Expansions in Eigenfunctions of Selfadjoint Operators}, Translations of Mathematical Monographs \textbf{17}, American Mathematical Society, Rhode Island, 1968.

\bibitem{Ch}\textrm{Chihara, T.S.}, \textit{An introduction to orthogonal polynomials}, Gordon and Breach Science Publishers, 1968.

\bibitem{Con} Conolly, B.W., {\em On randomized random walks}, SIAM Rev.\textbf{13} (1971), 81--99.

\bibitem{DIW} Dai, D., Ismail, M.E.H. and Wang, X., {\em Doubly infinite Jacobi matrices revisited: resolvent and spectral measure}, Adv. Math. \textbf{343}  (2019), 157--192.

\bibitem{DRSZ} Dette, H., Reuther, B., Studden, W. and Zygmunt, M., {\em Matrix measures and random walks with a block tridiagonal transition matrix}, SIAM J. Matrix Anal. Applic. \textbf{29} (2006), 117--142.

\bibitem{dCIM} Di Crescenzo, A., Iuliano, A. and Martinucci, B., {\em On a bilateral birth-death process with alternating rates}, Ric. Mat. \textbf{61} (2012), 157--169.

\bibitem{MDIB} Dom\'inguez de la Iglesia, M., {\em Orthogonal polynomials in the spectral analysis of Markov processes. Birth-death models and diffusion}, Encyclopedia of Mathematics and its Applications \textbf{181}, Cambridge University Press, 2021.

\bibitem{DG1}\textrm{Dur\'an, A.J. and Gr\"unbaum, F.A.},
\textit{Orthogonal matrix polynomials satisfying second order
differential equations}, Internat. Math. Research Notices, 2004:
\textbf{10} (2004), 461--484.


\bibitem{DG2}\textrm{Dur\'an, A.J. and Gr\"unbaum, F.A.},
\textit{A survey on orthogonal matrix polynomials satisfying second order differential equations}, J. Comput. Appl. Math. \textbf{178} (2005), 169--190.

\bibitem{DdI2} Dur\'an, A.J. and de la Iglesia, M.D., \emph{Second order differential operators having several families of orthogonal matrix polynomials as eigenfunctions}, Int. Math. Res. Not. \textbf{2008} (2008) rnn084.


\bibitem{G2}\textrm{Gr\"unbaum, F.A.},
\textit{Random walks and orthogonal polynomials: some challenges}, Probability, Geometry and Integrable Systems, MSRI Publication, volumen \textbf{55}, 2007.

\bibitem{G1}\textrm{Gr\"unbaum, F.A.}, \emph{QBD processes and matrix orthogonal polynomials: some new explicit examples}, Numerical Methods for Structured Markov Chains, eds. D. Bini, B. Meini, V. Ramaswami, M.A. Remiche and P. Taylor, Dagstuhl Seminar Proceedings, 2008.

\bibitem{G4} \textrm{Gr\"unbaum, F.A.}, \emph{An urn model associated with Jacobi polynomials}, Commun. Applied Math. Comput. Sciences \textbf{5} (2010), 55--63.

\bibitem{G3}\textrm{Gr\"unbaum, F.A.}, \emph{The Darboux process and a noncommutative bispectral problem: some explorations and challenges}, in E.P. van den Ban and J.A.C. Kolk (eds.), \emph{Geometric Aspects of Analysis and Mechanics: In Honor of the 65th Birthday of Hans Duistermaat}, Progress in Mathematics 292, Springer, 2011.

\bibitem{GH0} Gr\"unbaum, F.A. and Haine, L., \emph{A theorem of Bochner, revisited}, in: Algebraic Aspects of Integrable Systems (A.S. Fokas and I.M.Gelfand, eds.), Progr. Nonlinear Differential Equations Appl., vol. 26, Birkhauser, Boston, 1997, 143--172.

\bibitem{GH} Gr\"unbaum, F.A. and Haine, L.,
\textit{Associated polynomials, spectral matrices and the bispectral problem}, Methods and Applications of Analysis \textbf{6} (1999), 209--224.


\bibitem{GdI3} Gr\"unbaum, F.A. and de la Iglesia, M.D.,
\textit{Stochastic LU factorizations, Darboux transformations and urn models}, J. Appl. Prob. \textbf{55} (2018), 862--886.

\bibitem{GdI4} Gr\"unbaum, F.A. and de la Iglesia, M.D.,
\textit{Stochastic Darboux transformations for quasi-birth-and-death processes and urn models}, J. Math. Anal. Appl. \textbf{478} (2019), 634--654.

\bibitem{GPT1} Gr\"unbaum, F.A., Pacharoni, I. and Tirao, J.A., {\em Matrix-valued spherical functions associated to the complex projective plane}, J. Functional Analysis {\bf 188} (2002), 350--441.

\bibitem{GPT2} Gr\"unbaum, F.A., Pacharoni, I. and Tirao, J.A., {\em A matrix-valued solution to Bochner's problem}, J. Physics A: Math. Gen. {\bf 34} (2001), 10647--10656.

\bibitem{dI3} de la Iglesia, M.D., \emph{Spectral analysis of bilateral birth-death processes: some new explicit examples}, Adv. Appl. Prob. \textbf{54} (2022), 1193--1221.

\bibitem{dIJ1} de la Iglesia, M.D. and Juarez, C., \emph{The spectral matrices associated with the stochastic Darboux transformations of random walks on the integers}, J. Approx. Theory \textbf{258} (2020), 105458.

\bibitem{dIJ2} de la Iglesia, M.D. and Juarez, C., \emph{Absorbing-reflecting factorizations for birth-death chains on the integers and their Darboux transformations}, J. Approx. Theory \textbf{266} (2021), 105583, 27 pp. 

\bibitem{ILVW} Ismail, M.E.H., Letessier, J., Valent, G. and Wimp, J.,  \emph{Associated Wilson polynomials}, Canad. J. Math. \textbf{42} (1990), 659--695.

\bibitem{ILMV}  Ismail, M.E.H., Letessier, J., Masson, D. and Valent, G., {\em Birth and death processes and orthogonal polynomials}, in Orthogonal Polynomials, P.~Nevai (editor),  Kluwer Acad. Publishers (1990), 229--255.

\bibitem{IM} Ismail, M.E.H. and Masson, D., {\em Two families of orthogonal polynomials related to Jacobi polynomials}, Rocky Mountain J. Math. \textbf{21} (1991), 359--375.

\bibitem{IR} Ismail, M.E.H. and Rahman, M., \emph{The associated Askey-Wilson polynomials}, Trans. Amer. Math.
Soc \textbf{328} (1991), 201--237.

\bibitem{IsmS} Ismail, M.E.H. and \v{S}tampach, F., {\em Spectral analysis of two doubly infinite Jacobi matrices with exponential entries}, J. Funct. Analysis \textbf{276}  (2019), 1681--1716.

\bibitem{KMc2}  Karlin, S. and McGregor, J., {\em The differential equations of birth and death processes, and the Stieltjes moment problem}, Trans. Amer. Math. Soc., \textbf{85} (1957), 489--546.

\bibitem{KMc3}  Karlin, S. and McGregor, J., {\em The classification of birth-and-death processes}, Trans. Amer. Math. Soc., \textbf{86} (1957), 366--400.

\bibitem{KMc6}  Karlin, S. and McGregor, J., {\em Random walks}, IIlinois J. Math., \textbf{3} (1959), 66--81.

\bibitem{LaR} Latouche, G. and Ramaswami, V., {\em Introduction to Matrix Analytic Methods in Stochastic Modeling}, ASA-SIAM Series on Statistics and Applied Probability, 1999.

\bibitem{MR}  Masson, D.R. and Repka, J., {\em Spectral theory of Jacobi matrices in $\ell^2(\mathbb{Z})$ and the $su$(1,1) Lie algebra}, SIAM J. Math. Anal., \textbf{22} (1991), 1131--1146.

\bibitem{MS} Matveev, V.B.  and Salle, M.A., \emph{Differential-difference evolution equations II: Darboux transformation for the Toda lattice}, Lett. Math. Phys. \textbf{3} (1979), 425--429.


\bibitem{Neu} Neuts, M.F., {\em Structured Stochastic Matrices of $M/G/1$ Type and Their Applications}, Marcel Dekker, New York, 1989.

\bibitem{Pal} Palama, G. \emph{Polinomi piu generali di altri classici e dei loro associati e relazioni tra essi funzioni
di seconda specie}, Riv. Mat. Univ. Parma \textbf{4} (1953), 363--383.

\bibitem{PL}  Parthasarathy, P.R. and Lenin, R.B., \emph{Birth and death process (BDP) models with applications: queueing, communication systems, chemical models, biological models: the state-of-the-art with a time-dependent perspective}. American Series in Mathematical and Management Sciences, vol. 51, American Sciences Press, Columbus, 2004.


\bibitem{PruT} Pruitt, W.E., {\em Bilateral birth and death processes}, Technical report, Applied Mathematics and Statistics Laboratories, Stanford University, California, 1960.

\bibitem{Pru} Pruitt, W.E., \emph{Bilateral birth and death processes}, Trans. Amer. Math. Soc. \textbf{107} (1962), 508--525.

\bibitem{TB} Tarabia, A.M.K. and El-Baz, A.H., \emph{Transient solution of a random walk with chemical rule}, Physica A \textbf{382} (2007), 430--438.

\bibitem{Wi} Wimp, J., \emph{Explicit formulas for the associated Jacobi polynomials and some applications}, Can. J. Math. \textbf{34} (1987), 983--1000.


\end{thebibliography}
\end{document}